\documentclass[11pt,twoside,reqno,centertags,draft]{amsart}
\usepackage{amsfonts}
\usepackage{color,enumitem,graphicx}
\usepackage[colorlinks=true,urlcolor=blue,
citecolor=red,linkcolor=blue,linktocpage,pdfpagelabels,
bookmarksnumbered,bookmarksopen]{hyperref}

  \usepackage{amsmath,amsthm,amsfonts,amssymb}
  \usepackage{ulem}
\begin{document}

\title{ Existence and concentration of ground states  to  fractional Choquard equations}
\date{}
\maketitle

\vspace{ -1\baselineskip}

{\small
\begin{center}
{\sc  Jianfu Yang} \\
Department of Mathematics,
Jiangxi Normal University\\
Nanchang, Jiangxi 330022,
People's Republic of China\\
email: jfyang200749@sina.com\\[10pt]

{\sc  Jinge Yang} \\
School of Science,
Jiangxi University of Water Resources and Electric Power\\
Nanchang, Jiangxi 330099,
People's Republic of China\\
email: Jinge Yang: jgyang2007@yeah.net\\[10pt]

\end{center}
}

\renewcommand{\thefootnote}{}

\begin{quote}
{\bf Abstract.}\ \ In this paper, we study the nonlinear fractional Choquard equation
\begin{equation*}
			(-\Delta)^su+Vu=(|x|^{-\gamma}*|u|^2)u \quad {\rm in} \quad \mathbb{R}^N,
\end{equation*}
where $0<\gamma<4$, $0<s<1$,  $N\geq 4$ and $V\in C^1(\mathbb{R}^N)$ is a positive potential. Set $s_0=\frac {\gamma}{4}$. Under suitable assumptions  on $V$, we  prove  that   the equation  admits a nonnegative ground state solution for $s\in(s_0,1)$, whereas no ground state solution exists for $0<s\le s_0$. Furthermore,  we  show that  any ground state solution $u_s$ blows up and concentrates at a minimum point of $V$ as $s\downarrow s_0$. Finally,  up to a subsequence, the ground state solution $u_s$ converges to a ground state solution of  the classical Choquard equation   as $s\uparrow 1$.

\end{quote}

\newcommand{\N}{\mathbb{N}}
\newcommand{\R}{\mathbb{R}}
\newcommand{\Z}{\mathbb{Z}}

\newcommand{\cA}{{\mathcal A}}
\newcommand{\cB}{{\mathcal B}}
\newcommand{\cC}{{\mathcal C}}
\newcommand{\cD}{{\mathcal D}}
\newcommand{\cE}{{\mathcal E}}
\newcommand{\cF}{{\mathcal F}}
\newcommand{\cG}{{\mathcal G}}
\newcommand{\cH}{{\mathcal H}}
\newcommand{\cI}{{\mathcal I}}
\newcommand{\cJ}{{\mathcal J}}
\newcommand{\cK}{{\mathcal K}}
\newcommand{\cL}{{\mathcal L}}
\newcommand{\cM}{{\mathcal M}}
\newcommand{\cN}{{\mathcal N}}
\newcommand{\cO}{{\mathcal O}}
\newcommand{\cP}{{\mathcal P}}
\newcommand{\cQ}{{\mathcal Q}}
\newcommand{\cR}{{\mathcal R}}
\newcommand{\cS}{{\mathcal S}}
\newcommand{\cT}{{\mathcal T}}
\newcommand{\cU}{{\mathcal U}}
\newcommand{\cV}{{\mathcal V}}
\newcommand{\cW}{{\mathcal W}}
\newcommand{\cX}{{\mathcal X}}
\newcommand{\cY}{{\mathcal Y}}
\newcommand{\cZ}{{\mathcal Z}}

\newcommand{\abs}[1]{\lvert#1\rvert}
\newcommand{\xabs}[1]{\left\lvert#1\right\rvert}
\newcommand{\norm}[1]{\lVert#1\rVert}

\newcommand{\loc}{\mathrm{loc}}
\newcommand{\p}{\partial}
\newcommand{\h}{\hskip 5mm}
\newcommand{\ti}{\widetilde}
\newcommand{\D}{\Delta}
\newcommand{\e}{\epsilon}
\newcommand{\bs}{\backslash}
\newcommand{\ep}{\emptyset}
\newcommand{\su}{\subset}
\newcommand{\ds}{\displaystyle}
\newcommand{\ld}{\lambda}
\newcommand{\vp}{\varphi}
\newcommand{\wpp}{W_0^{1,\ p}(\Omega)}
\newcommand{\ino}{\int_\Omega}
\newcommand{\bo}{\overline{\Omega}}
\newcommand{\ccc}{\cC_0^1(\bo)}
\newcommand{\iii}{\opint_{D_1}D_i}

\theoremstyle{plain}
\newtheorem{Thm}{Theorem}[section]
\newtheorem{Lem}[Thm]{Lemma}
\newtheorem{Def}[Thm]{Definition}
\newtheorem{Cor}[Thm]{Corollary}
\newtheorem{Prop}[Thm]{Proposition}
\newtheorem{Rem}[Thm]{Remark}
\newtheorem{Ex}[Thm]{Example}

\numberwithin{equation}{section}
\newcommand{\meas}{\rm meas}
\newcommand{\ess}{\rm ess} \newcommand{\esssup}{\rm ess\,sup}
\newcommand{\essinf}{\rm ess\,inf} \newcommand{\spann}{\rm span}
\newcommand{\clos}{\rm clos} \newcommand{\opint}{\rm int}
\newcommand{\conv}{\rm conv} \newcommand{\dist}{\rm dist}
\newcommand{\id}{\rm id} \newcommand{\gen}{\rm gen}
\newcommand{\opdiv}{\rm div}

\vskip 0.2cm \arraycolsep1.5pt
\newtheorem{Lemma}{Lemma}[section]
\newtheorem{Theorem}{Theorem}[section]
\newtheorem{Definition}{Definition}[section]
\newtheorem{Proposition}{Proposition}[section]
\newtheorem{Remark}{Remark}[section]
\newtheorem{Corollary}{Corollary}[section]

\section{Introduction}

\bigskip

This paper is concerned with  the existence,  nonexistence  and the concentration behavior of ground state solutions for the following nonlinear fractional Choquard equation
	\begin{equation}\label{eq:1.1}\tag{$P_s$}
		(-\Delta)^su+Vu=(|x|^{-\gamma}*|u|^2)u,\quad {\rm in} \quad \mathbb{R}^N,
	\end{equation}
	where  $0<\gamma<4$, $0<s<1$, $N\geq 4$, and $V\in C^1(\mathbb{R}^N)$ is positive. The fractional Laplacian $(-\Delta)^s$  is defined on the Schwartz space  via the Fourier transform
	\[
	\mathcal{F}[(-\Delta)^su](\xi)=|\xi|^{2s} \mathcal{F}(u)(\xi),
	\]
	where
	\[
	\mathcal{F}(u)(\xi)=\int_{\mathbb{R}^N}e^{- ix\cdot \xi}u(x)\,dx
	\]
	denotes the Fourier transform of $u$.

In this work, we regard $s$ as a parameter, which is motivated by the study of semiclassical Schr\"odinger equations
	\begin{equation}\label{eq:1.3a}
		\begin{cases}
			-\varepsilon^2\Delta u +V(x)u=|u|^{p-2}u,\quad x\in\mathbb{R}^N,\\
			u(x)\to 0,\quad  |x|\to \infty.
		\end{cases}
	\end{equation}
The concentration behavior of ground states of equation \eqref{eq:1.3a} as $\varepsilon\to 0$ has been widely investigated for various classes of potentials. It is well known that concentration points are closely related to the critical points of $V$; see, for instance, \cite{FW, O}.  For further results in this direction, we refer to  \cite{ABC, AMS, BPW, BO,  PF, PF1, G, W93}  and the references therein.

Similar results have been established for fractional Schr\"odinger equations
	\begin{equation}\label{eq:1.3c}
		(-\Delta)^s u +V(x)u=|u|^{2^*_s-2-\varepsilon}u,\quad x\in\mathbb{R}^N,
	\end{equation}
where $2^*_s=\frac{2n}{n-2s}$ is the fractional Sobolev critical exponent.
The case $s=1$ was studied in \cite{PW,W}, and the fractional case $s\in (0,1)$
 was considered in \cite{CW}. Vilasi and Wang \cite{VW} investigated the fractional Choquard equation
\begin{equation}\label{fc1}
		(-\Delta)^s u +u=A_\alpha(|x|^{-(N-\alpha)}*|u|^{2_{\alpha,s}^*-\varepsilon})|u|^{2_{\alpha,s}^*-\varepsilon-2}u,
	\end{equation}
	where $2_{\alpha,s}^*=\frac{N+\alpha}{N-2s}$.
	Recently,  Yang and Yang \cite{YY}  used $s$
 as a parameter to study the concentration of ground states for nonlinear fractional elliptic problems
	\begin{equation}\label{fs}
		(-\Delta)^s u +V(x)u=|u|^{p}u,\quad x\in\mathbb{R}^N.
	\end{equation}

Motivated by \cite{YY}, we extend the related results to the fractional Choquard equation \eqref{eq:1.1}, which contains both nonlocal fractional operators and nonlocal convolution nonlinearities.

\bigskip

Throughout the paper, we assume the following condition on $V\in C^1(\mathbb{R}^N)$:
\bigskip

	$(V)$ $0<V_0\leq V(x)<V_\infty:=\lim_{x\to\infty}V(x)<+\infty$, and $x\cdot\nabla V$ is bounded on
	$\mathbb{R}^N$.
\bigskip

Under assumption $(V)$, for any $u\in H^s(\mathbb{R}^N)$, we have
	\begin{equation}\label{svh1}
		\min\{1,V_0\}\|u\|_{H^s(\mathbb{R}^N)}^2\leq \|u\|_{s,V}^2	\leq \max\{1,V_\infty\}\|u\|_{H^s(\mathbb{R}^N)}^2,
	\end{equation}
where
\[
\|u\|_{H^s(\mathbb{R}^N)}^2=\int_{\mathbb{R}^N}|(-\Delta)^{\frac s2}u(x)|^2\,dx+\int_{\mathbb{R}^N}|u(x)|^2\,dx,
\]
\[
	\|u\|_{s,V}^2=\int_{\mathbb{R}^N}|(-\Delta)^{\frac s2}u(x)|^2\,dx+\int_{\mathbb{R}^N}V(x)|u(x)|^2\,dx.
	\]
Thus, $\|u\|_{s,V}$ is an equivalent norm of $H^s(\mathbb{R}^N)$.

By the Hardy-Littlewood-Sobolev inequality and fractional Sobolev embedding,
	\[
	D_\gamma(u,u):=\iint_{\mathbb{R}^N\times\mathbb{R}^N}\frac{|u(x)|^2|u(y)|^2}{|x-y|^\gamma}dxdy
\leq C\|u\|_{H^s}^4,
	\]
which  implies that the functional
\begin{equation}\label{fu}
		I_{s,\gamma}^V(u)=\frac{\|u\|_{s,V}^2}{D_\gamma(u,u)^{\frac12}}
	\end{equation}
is well-defined on $H^s(\mathbb{R}^N)\setminus\{0\}$. The existence of ground state solutions of \eqref{eq:1.1} is equivalent to the attainability of the minimization problem
 \begin{equation}\label{ssgv}
		S_{s,\gamma}^V=\inf_{u\in H^s(\mathbb{R}^N\setminus \{0\}\}}	I_{s,\gamma}^V(u).
	\end{equation}

A ground state solution to \eqref{eq:1.1}
 is a solution with minimal energy. By Lemma 2.6 in \cite{YY}, a function $u\in H^s(\mathbb{R}^N)$ is a ground state of $(P_s)$ if and only if it is a minimizer of $S_{s,\gamma}^V$ and solves $(P_s)$ with $\|u\|_{s,V}= S_{s,\gamma}^V$.

Let $s_0=\frac {\gamma}{4}$.  We  show that  ground states exist for $s\in (s_0,1)$, and do not exist for $s\in (0,s_0]$. For the constant potential $V(x)\equiv1$, related results have been established in \cite{AS}.

\bigskip

For $s\in (s_0,1)$, our main results are as follows.
	\begin{Theorem}\label{thm1} Let $s_0<s<1$and assumption $(V)$ hold. Then
		
		$(i)$ Equation \eqref{eq:1.1} admits a nonnegative ground state solution $u_s$;
		
		$(ii)$ Any nonnegative ground state solution $u_s$ belongs to $C(\mathbb{R}^N)$ and satisfies the decay estimate
		$$0\leq u_s(x)\leq C(N,s,\gamma,V_0,u_s)(1+|x|^2)^{-\frac{N+2s}{2}}.$$
	\end{Theorem}

\bigskip

If $0<s\leq s_0$, problem $(P_{s})$ is of critical or supercritical type with no attainable minimizer.
	\begin{Theorem}\label{p1}
	  Assume  $0<s\leq s_0$ and assumption $(V)$ holds. Then the infimum $S_{s,\gamma}^V$ is not attained in $H^s(\mathbb{R}^N)\setminus\{0\}$.
	\end{Theorem}

\bigskip

	Fix $\gamma\in (0,4)$. 	Let $u_s$ be a nonnegative ground state  of  \eqref{eq:1.1} and $x_s$ be a maximum point of $u_s$. We study the blow-up and concentration behavior of $u_s$  as  $s\downarrow s_0$.  The proof adopts the framework of \cite{PW}, while new delicate estimates are required due to the double nonlocal structure of \eqref{eq:1.1}.
	\begin{Theorem}\label{thm2}
		There exists a constant $C(N,s_0)>0$ such that
		\begin{flushleft}
			${(i)}$ the limit $$\lim_{s\downarrow s_0}	\|u_s\|_{L^\infty(\mathbb{R}^N)}^{-1}u_s\Big(\|u_s\|_{L^\infty(\mathbb{R}^N)}^{-\frac{2}{N+2s-\gamma}}x+x_s\Big)=\Big(1+\frac{|x|^2}{C(N,s_0)}\Big)^{-\frac{N-2s_0}{2}}
	$$ holds in $\dot{H}^{s_0}(\mathbb{R}^N)\cap C(\mathbb{R}^N)$;
		\end{flushleft}
		\smallskip
		\begin{flushleft}
			${(ii)}$ $\lim_{s\downarrow s_0}V(x_s)=\inf_{x\in \mathbb{R}^N}V(x);$
		\end{flushleft}
		\smallskip
		$(iii)$ there exists  $A(n,s_0)>0$ such that
		\[
		\lim_{s\downarrow s_0}(s-s_0)\|u_s\|_\infty^{\frac{4s_0}{N-2s_0}}=A(n,s_0)\inf_{x\in \mathbb{R}^N}V(x).
		\]
	\end{Theorem}	

\bigskip

Finally, we study the asymptotic behavior of ground states $u_s$ as $s\uparrow 1$.

\begin{Theorem}\label{thm3}For any sequence $s_k\uparrow 1$, there exist a subsequence (still denoted by $s_k$) and a ground state solution $u$ of $(P_s)$ with $s=1$ such that the sequence $\{u_{s_k}\}$ of corresponding ground state solutions satisfies
\[
\lim_{k\to\infty} u_{s_k}=u \quad {\rm in} \quad L^{2_{s_0}^*}(\mathbb{R}^N)
\]
up to subsequence.
\end{Theorem}	
	
	The rest of the paper is organized as follows. Section 2 presents  some preliminary lemmas. Section 3 gives the proofs of Theorems \ref{thm1} and \ref{p1}.
Section 4 is devoted to the proof of Theorem \ref{thm2}.
Section 5 proves Theorem \ref{thm3}.
For simplicity, we denote $\|\cdot\|_p=\|\cdot\|_{L^p(\mathbb{R}^N)}$.

	\bigskip
	
	\bigskip

	\section{Preliminaries}

\bigskip

	 In this section, we collect some useful lemmas that will be used throughout the paper.
	\begin{Lemma}[Hardy--Littlewood--Sobolev inequality]\label{lemb} Let $t,r>1$ and $0<\mu<N$. For any  $f\in L^t(\mathbb{R}^N)$ and $h\in L^r(\mathbb{R}^N)$,  there exists a constant $C=C(t,N,\mu,r)>0$ such that
		\[
		\int_{\mathbb{R}^N}\int_{\mathbb{R}^N}\frac{f(x)h(y)}{|x-y|^\mu}\,dxdy\leq C(t,N,\mu,r)
		\|f\|_{L^t(\mathbb{R}^N)}\|h\|_{L^r(\mathbb{R}^N)},
		\]
		provided $\frac1t+\frac\mu N+\frac1r=2$.
	\end{Lemma}

\bigskip

Lemma \ref{lemb} implies the following result.

	\begin{Lemma}\label{lemb2} Let $0<\gamma<N$ and $p>1$. For any $f\in L^p(\mathbb{R}^N)$, we have
		$|x|^{-\gamma}*f\in L^q(\mathbb{R}^N)$ and
		\[
		\Big\||\cdot|^{-\gamma}*f\Big\|_{L^q(\mathbb{R}^N)}\leq C(p,N,\gamma)\|f\|_{L^p(\mathbb{R}^N)},
		\]
		where $1+\frac1q=\frac1p+\frac\gamma N$.
	\end{Lemma}
	
\bigskip

	The following  fractional Sobolev theorem can be found,  for example in  \cite[Theorem 1.1.8]{Am2021} and \cite[Theorem 1.1]{CT}.
	
	\begin{Lemma}[Fractional Sobolev embedding]\label{lemc} Let $N\geq 3$ and $2_s^*=\frac {2N}{N-2s}$. Then there exists a sharp constant $S_*=S(N,s)>0$ such that
		\[
		S_*\|u\|_{L^{2_s^*}(\mathbb{R}^N)}\leq \|(-\Delta)^{\frac s2}u\|_{L^2(\mathbb{R}^N)}
		\]
		for any $u\in \dot{H}^{s}(\mathbb{R}^N)$. Moreover,  $0<C_1(N)\leq S_*\leq C_2(N)$.  The space ${H}^{s}(\mathbb{R}^N)$ is continuously embedded in $L^q(\mathbb{R}^N)$ for  $q\in[2,2_s^*]$ and compactly embedded in $L_{loc}^q(\mathbb{R}^N)$ for   $q\in[2,2_s^*)$.
	\end{Lemma}

\bigskip
	
	It can be founded in \cite[Page 5]{MS} the following result.
	\begin{Lemma}\label{lema} Let $s\in (0,1)$, $N>2s$, $\mu\in (0,N)$ and $2_{\mu,s}^*=\frac{2N-\mu}{N-2s}$. Then
		\[
		S_{s,\mu}^H\Big(\int_{\mathbb{R}^N}\int_{\mathbb{R}^N}
		\frac{|u(x)|^{2_{\mu,s}^*}|u(y)|^{2_{\mu,s}^*}}{|x-y|^\mu}\,\,dxdy\Big)^{\frac{1}{2_{\mu,s}^*}}
		\leq \|(-\Delta)^{\frac s2}u\|_{L^2(\mathbb{R}^N)}^2.
		\]
		The best constant $S_{s,\mu}^H$ is attained by $u$ if and only if
		\[
		u(x)=C\Big(\frac{t}{t^2+|x-x_0|^2}\Big)^{\frac{N-2s}{2}}
		\]
		for some $x_0\in \mathbb{R}^N$, $C>0$ and $t>0$. In particular, for $\mu=4s$,
		\[
		S_{s,4s}^H\Big(\int_{\mathbb{R}^N}\int_{\mathbb{R}^N}
		\frac{|u(x)|^{2}|u(y)|^{2}}{|x-y|^{4s}}\,\,dxdy\Big)^{\frac{1}{2}}
		\leq \|(-\Delta)^{\frac s2}u\|_{L^2(\mathbb{R}^N)}^2.
		\]
	\end{Lemma}

\bigskip
	
	We denote $S_s^H:=S_{s,4s}^H$ for brevity.
	
	\bigskip
	Now we recall the uniform Schauder and Harnack estimates for fractional equations in \cite{YY}.
	\begin{Lemma}\label{lemd}(Lemma 3.4 in \cite{YY}) Let $0<s_1<s<1$. If $(-\Delta)^su=f$ with $u,f\in L^\infty(\mathbb{R}^N)$, then there exists $\alpha\in (0,s_1)$ such that
		\[
		\|u\|_{C^{0,\alpha}(\mathbb{R}^N)}
		\leq C(N,\alpha,s_1)\Big(\|f\|_{L^\infty(\mathbb{R}^N)}+\|u\|_{L^\infty(\mathbb{R}^N)}\Big).
		\]
	\end{Lemma}

\bigskip

	\begin{Lemma}\label{leme2}(Lemma 3.6 in \cite{YY})
		Let $N\geq3$, $r>0$, $0<s<1$ and $y_0\in \mathbb{R}^N$. Suppose that $u\in H^s(\mathbb{R}^N)$ satisfies
		\begin{equation}\label{ulau}		
(-\Delta)^su\leq a(x)u,\quad x\in\mathbb{R}^N.
\end{equation}
		There exist $\delta=\delta(N)>0$ and $C=C(N,r)>0$ such that if
		\[\int_{B_{2r}(y_0)}|a(x)|^{\frac{N}{2s}}\,dx\leq \delta,\]
		then
		\[
		\|u\|_{L^{\frac{(2_s^*)^2}{2}}(B_r(y_0))}\leq C(N,r)\|u\|_{L^{2_s^*}(\mathbb{R}^N)}.
		\]
	\end{Lemma}

\bigskip

	\begin{Lemma}\label{leme}(Lemma 3.1 in \cite{YY}) Let $0<r<1$, $y_0\in \mathbb{R}^N$ and $u\in H^s(\mathbb{R}^N)$ be a nonnegative function
		satisfying
		\[
		(-\Delta)^su=a(x)u+b(x),\quad x\in B_r(y_0).
		\]
		Assume that $0<s_1\leq s\leq s_2<1$, $a,b\in L^q(B_r(y_0))$ with $q>\frac{N}{2s_1}$. Let $C_a$
		be an upper bound of $\|a\|_{L^q(B_r(y_0))}$. Then
		\[
		\sup_{B_{r/2}(y_0)}u\leq C_2(N,s_1,s_2,C_a)\Big(\inf_{B_{r/2}(y_0)}u+\|b\|_{L^q(B_r(y_0))}\Big).
		\]
	\end{Lemma}

\bigskip

\begin{Lemma}\label{lem:A.3}(Lemma 3.7 in \cite{YY})
 Let $0<s<1$ and $u\in H^s(\mathbb{R}^n)$  be a nonnegative solution of \eqref{ulau}.
If $a\in L^t(\mathbb{R}^n)$ with $t>\frac{n}{2s}$ and
\[
\|a\|_{L^t(\mathbb{R}^n)}\leq C_a,
\]
then
\[
\|u\|_\infty\leq C(n,t,s,C_a)\|u\|_{2_s^*}.
\]
Moreover, for fixed $s_0\in (0,1)$, if $t>\frac{n}{2s_0}$ and $C_a$ is independent of $s$, $C(n,t,s, C_a)$ is uniformly bounded for $s\in(s_0,1)$.
\end{Lemma}

\bigskip

We remark that for a fixed $s$, Lemmas \ref{lemd} and \ref{leme} were obtained in \cite{S} and \cite{JL} respectively.
	
Finally, in order to study the limiting behavior of the ground state solution of \eqref{eq:1.1} as $s\downarrow s_0$, we have the following results.
\bigskip

	\begin{Lemma}\label{lemf}
		Let $0<s_1\leq s_2\leq 1$ and $u\in H^{s_2}(\mathbb{R}^N)$. Then
		\begin{flushleft}
			${(i)}\quad \|u\|_{H^{s_1}(\mathbb{R}^N)}^2\leq 2\|u\|_{H^{s_2}(\mathbb{R}^N)}^2;$
		\end{flushleft}	
		\begin{flushleft}
			${(ii)}\quad  \|u\|_{\dot{H}^{s_1}(\mathbb{R}^N)}^2\leq \|u\|_{\dot{H}^{s_2}(\mathbb{R}^N)}^{\frac{2s_1}{s_2}}
			\|u\|_{L^{2}(\mathbb{R}^N)}^{2(1-\frac{s_1}{s_2})}$.
		\end{flushleft}	
	\end{Lemma}	
	\begin{proof}
		For $u\in H^{s_2}(\mathbb{R}^N)$, the conclusion $(i)$ follows from the inequality
		\begin{equation*}
			\begin{split}
				\|u\|_{H^{s_1}(\mathbb{R}^N)}^2&=\int_{\mathbb{R}^N}(1+|\xi|^{2s_1})|\hat{u}(\xi)|^2\,d\xi
				\leq\int_{\mathbb{R}^N}(2+|\xi|^{2s_2})|\hat{u}(\xi)|^2\,d\xi \leq2\|u\|_{H^{s_2}(\mathbb{R}^N)}^2;
			\end{split}
		\end{equation*}
		whereas for $(ii)$, we deduce it  from the inequality
		\begin{equation*}
			\begin{split}
				\|u\|_{\dot{H}^{s_1}(\mathbb{R}^N)}^2&=\int_{\mathbb{R}^N}|\xi|^{2s_1}|\hat{u}(\xi)|^2\,d\xi	\\	
				&\leq \Big(\int_{\mathbb{R}^N}|\xi|^{2s_2}|\hat{u}(\xi)|^2\,d\xi\Big)^{\frac{s_1}{s_2}}
				\Big(\int_{\mathbb{R}^N}|\hat{u}(\xi)|^2\,d\xi\Big)^{1-\frac{s_1}{s_2}}\\
				&=\|u\|_{\dot{H}^{s_2}(\mathbb{R}^N)}^{\frac{2s_1}{s_2}}
				\|u\|_{L^2(\mathbb{R}^N)}^{2(1-\frac{s_1}{s_2})}.
			\end{split}
		\end{equation*}	
	\end{proof}

\bigskip
	
	\section{Existence and nonexistence}

\bigskip

	In this section, we will show the existence and nonexistence of ground state solutions of problem \eqref{eq:1.1}. That is, we will prove  Theorem \ref{thm1} and
Theorem \ref{p1}.

\bigskip

	{\bf Proof of Theorem \ref{thm1}} $(i)$ First we show that $S_{s,\gamma}^V$ given in \eqref{ssgv} is positive. Indeed,  since $0<\gamma<4s$, $\frac{4N}{2N-\gamma}\leq 2_s^*$.
	By \eqref{svh1}, Lemmas \ref{lemb} and \ref{lemc}, we derive
	\begin{equation}\label{sgvd0}
		\begin{split}
			D_\gamma(u,u)\leq C(N,\gamma)\|u\|_{\frac{4N}{2N-\gamma}}^4\leq C(N,\gamma,s,V)\|u\|_{s,V}^4.
		\end{split}
	\end{equation}
This  implies that $S_{s,\gamma}^V>0$.
\bigskip
	
	Next, we claim that $S_{s,\gamma}^{V_\infty}>S_{s,\gamma}^{V}$.

We know from \cite{AS} that there is a positive minimizer  $w_s$ of $S_{s,\gamma}^{V_\infty}$.
	By the assumption $(V)$, we have
	\begin{equation*}
		\begin{split}
			S_{s,\gamma}^{V_\infty}&=\frac{\|w_s\|_{s,V_\infty}^2}{D_\gamma(w_s,w_s)^{\frac12}}>\frac{\|w_s\|_{s,V}^2}{D_\gamma(w_s,w_s)^{\frac12}}\geq S_{s,\gamma}^{V},
		\end{split}
	\end{equation*}
	that is $S_{s,\gamma}^{V_\infty}>S_{s,\gamma}^{V}$.
	
	\bigskip
Now we show that $S_{s,\gamma}^V$ is achieved.

	Let $\{w_n\}\subset H^s(\mathbb{R}^N)$ be a minimizing sequence of $S_{s,\gamma}^V$. Hence,
	$$I_{s,\gamma}^V(w_n)\to S_{s,\gamma}^V, \quad{\rm as}\quad n\to\infty.$$
 Choose
	\[
	u_n=\Big(\frac{(S_{s,\gamma}^V)^2}{D_\gamma(w_n,w_n)}\Big)^{\frac14}w_n,
	\]
	we may verify that
	\begin{equation}\label{unfc}
		D_\gamma(u_n,u_n)=\Big(S_{s,\gamma}^V\Big)^2\quad \text{and}\quad \lim_{n\to\infty}\|u_n\|_{s,V}^2=\Big(S_{s,\gamma}^V\Big)^2.
	\end{equation}	
	Thus, $\{u_n\}\subset H^s(\mathbb{R}^N)$ is  a bounded minimizing sequence of $S_{s,\gamma}^V$.
	By Lemma \ref{lemc}, we may assume that
	\begin{equation}\label{unwc}
		u_n\rightharpoonup u\quad \text{weakly in } \quad H^s(\mathbb{R}^N) \quad \text{as } \quad n\to\infty;
	\end{equation}	
	\begin{equation}\label{unlc}
		u_n\to u\quad \text{strongly in } \quad L_{loc}^2(\mathbb{R}^N) \quad \text{as } \quad n\to\infty.
	\end{equation}	
	Let $l=D_\gamma(u,u)$. The Fatou Lemma yields that $$0\leq l\leq \Big(S_{s,\gamma}^V\Big)^2.$$

Set
	$v_n=u_n-u$. Then,
\begin{equation}\label{unwca}
		v_n\rightharpoonup 0\quad \text{weakly in } \quad H^s(\mathbb{R}^N) \quad \text{as } \quad n\to\infty;
	\end{equation}	
	\begin{equation}\label{unlcb}
		v_n\to 0\quad \text{strongly in } \quad L_{loc}^2(\mathbb{R}^N) \quad \text{as } \quad n\to\infty.
	\end{equation}

  By $(V)$,
  \[
  \lim_{n\to\infty}\int_{\mathbb{R}^N}(V(x)-V_\infty)v_n\,dx =0.
  \]
Therefore,
	\begin{equation*}
		\begin{split}
			\Big(S_{s,\gamma}^V\Big)^2	=\lim_{n\to\infty}\|u_n\|_{s,V}^2=\lim_{n\to\infty}\|v_n\|_{s,V}^2+\lim_{n\to\infty}\|u\|_{s,V}^2=\lim_{n\to\infty}\|v_n\|_{s,V_\infty}^2+\lim_{n\to\infty}\|u\|_{s,V}^2
		\end{split}
	\end{equation*}
This yields from the definitions $S_{s,\gamma}^{V}$ and $S_{s,\gamma}^{V_\infty}$ that
\begin{equation}\label{unlcc}
\Big(S_{s,\gamma}^V\Big)^2	\geq S_{s,\gamma}^{V_\infty}\lim_{n\to\infty}\Big[D_\gamma(v_n,v_n)\Big]^{\frac12}
			+S_{s,\gamma}^V\Big[D_\gamma(u,u)\Big]^{\frac12}
\end{equation}
	
On the other hand, by the nonlocal Brezis-Lieb Lemma \cite[Lemma 2.2]{BFV} we deduce that
	\begin{equation*}
		\begin{split}
			\Big(S_{s,\gamma}^V\Big)^2	&=\lim_{n\to\infty}D_\gamma(u_n,u_n)=\lim_{n\to\infty}D_\gamma(v_n,v_n)+D_\gamma(u,u)
			=\lim_{n\to\infty}D_\gamma(v_n,v_n)+l.
		\end{split}
	\end{equation*}
	So \eqref{unlcc} becomes
	\begin{equation}\label{unlcd}
		\begin{split}
			\Big(S_{s,\gamma}^V\Big)^2&\geq 	
			S_{s,\gamma}^{V_\infty}\Big[\Big(S_{s,\gamma}^V\Big)^2-l\Big]^{\frac12}
			+S_{s,\gamma}^Vl^{\frac12},
		\end{split}
	\end{equation}
	which implies that $l=\Big(S_{s,\gamma}^V\Big)^2$. Indeed, were it not the case, we would deduce  from \eqref{unlcd}, the fact $l<\Big(S_{s,\gamma}^V\Big)^2$ and
$S_{s,\gamma}^{V_\infty}>S_{s,\gamma}^{V}$ that
	\[
	\Big(S_{s,\gamma}^V\Big)^2>S_{s,\gamma}^{V}\Big[\Big(S_{s,\gamma}^V\Big)^2-l\Big]^{\frac12}
	+S_{s,\gamma}^Vl^{\frac12}\geq \Big(S_{s,\gamma}^V\Big)^2,
	\]
	which is a contradiction. Hence, $l=D_\gamma(u,u)=\Big(S_{s,\gamma}^V\Big)^2$.
	
	 By \eqref{unfc} and \eqref{unwc},
	\[
	I_{s,\gamma}^V(u)= \frac{\|u\|_{s,V}^2}{D_\gamma(u,u)^{\frac12}}\leq \frac{\lim_{n\to\infty}\|u_n\|_{s,V}^2}{D_\gamma(u,u)^{\frac12}}=S_{s,\gamma}^V.
	\]

	The definition of $S_{s,\gamma}^V$ imlies  $I_{s,\gamma}^V(u)\geq S_{s,\gamma}^V$. Consequently,

	 $$I_{s,\gamma}^V(u)= S_{s,\gamma}^V.$$
Namely, $u$ is a minimizer of $S_{s,\gamma}^V$.
	Since $\big\||u|\big\|_{s,V}\leq \|u\|_{s,V}$, it follows that $|u|$ is also a minimizer of $S_{s,\gamma}^V$ and
	\[
	\Big[\frac d{dt}\Big(I_{s,\gamma}^V(|u|+t\varphi)\Big)\Big](0)=0,\quad \text{for any }\varphi\in C_c^\infty(\mathbb{R}^N).
	\]
	This yields that $|u|$ solves \eqref{eq:1.1}. We therefore conclude that $|u|$ is a ground state
	solution of \eqref{eq:1.1}.
	
	\bigskip
	$(ii)$ We start with the regularity of $u_s$. Applying Lemma \ref{lemd}  to \eqref{eq:1.1}, we only need  to show that $u_s$
belongs to $L^\infty(\mathbb{R}^N)$.

By \eqref{eq:1.1}, we have
\[
(-\Delta)^su_s\leq (|x|^{-\gamma}*|u_s|^2)u_s,\quad {\rm in} \quad \mathbb{R}^N.
\]

From Lemmas  \ref{lemb2} and \ref{lemc}, we get
	\begin{equation}\label{jj1}
		\begin{split}
			\big\||\cdot|^{-\gamma}*|u_s|^2\big\|_{\frac {2N}\gamma}
			\leq C(N,\gamma)\|u_s\|_{2_{s_0}^*}^2\leq C(N,\gamma,s)\|u_s\|_{H^s(\mathbb{R}^N)}^2.
		\end{split}
	\end{equation}
	
Since $\frac {2N}\gamma>\frac{N}{2s}$, taking $a(x)=|x|^{-\gamma}*|u_s|^2$ and $t=\frac {2N}\gamma$ in Lemma \ref{lem:A.3}, we obtain
\begin{equation}\label{usw}
\begin{split}
		\|u_s\|_\infty&\leq C_1(N,\gamma,s,\|u_s\|_{H^s(\mathbb{R}^N)})\|u_s\|_{2_s^*}
\leq C(N,\gamma,s,\|u_s\|_{H^s(\mathbb{R}^N)}).
\end{split}
	\end{equation}

	As a result, $u_s\in C(\mathbb{R}^N)$.

\bigskip

Next, we establish the decay law for $u_s$.	
	
Choose $\delta=\delta(N,s,\gamma,V_0,u_s)>0$  such that
	\[
	\|u_s\|_\infty^2\int_{\{y:|y|\leq \delta\}}|y|^{-\gamma}\,dy<\frac16V_0,
	\]
	and then choose $R=R(N,s,\gamma,V_0, u_s)>0$ so that
	\[
	\Big(\frac R2\Big)^{-\gamma}\|u_s\|_{H^s(\mathbb{R}^N)}^2<\frac16V_0,\quad
	\delta^{-\gamma}\int_{\{y:|y|\geq \frac R2\}}|u_s|^2\,dx<\frac16V_0.
	\]
	Then, if $|x|\geq R$ we have
	\begin{equation*}
		\begin{split}
			|x|^{-\gamma}*|u_s|^2&=\Big[\int_{\{y: |y|\leq R/2\}}+\int_{\{y: |y|\geq R/2,\,|x-y|\leq \delta\}}+
			\int_{\{y: |y|\geq R/2,\,|x-y|\geq \delta\}}\Big]
			\frac{|u_s(y)|^2}{|x-y|^{\gamma}}\,dy\\
			&\leq \Big(\frac R2\Big)^{-\gamma}\|u_s\|_{H^s(\mathbb{R}^N)}^2+\|u_s\|_\infty^2\int_{\{y:|y|\leq \delta\}}\frac 1{|y|^{\gamma}}\,dy+\frac 1{\delta^{\gamma}}\int_{\{y:|y|\geq \frac R2\}}|u_s|^2\,dx<\frac12V_0.
		\end{split}
	\end{equation*}
	
	Rewrite \eqref{eq:1.1} as
	\[
	(-\Delta)^su_s+\frac12V_0u_s=\Big[|x|^{-\gamma}*|u_s|^2-V(x)+\frac12V_0\Big]u_s:=f_s.
	\]

	For $|x|\geq R$,  since $f_s(x)\leq \Big[|x|^{-\gamma}*|u_s|^2-\frac12V_0\Big]u_s\leq 0$, we have
	\[
	{u}_s=\int_{\mathbb{R}^N}K_{s,\frac12V_0}(x-y)f_s(y)\,dy\leq \int_{\{y:|y|\leq R\}}K_{s,\frac12V_0}(x-y)f_s(y)\,dy,
	\]
	where $K_{s,\frac12V_0}>0$ is the fundamental solution of the equation
	\[
	(-\Delta)^s u+\frac12V_0 u=0.
	\]
	It is known from $(iii)$ of Lemma C.1 in \cite{FLS} that
	\[
	0\leq K_{s,\frac12V_0}\leq C(N,s,V_0)|x|^{-(N-2s)}.
	\]
	Thus, for $|x|\geq 2R$, we have
	\begin{equation*}
		\begin{split}
			{u}_s&\leq \int_{\{y:|y|\leq R\}}|x-y|^{-(N-2s)}f_s(y)\,dy\\
			&\leq \Big(\frac{|x|}2\Big)^{-(N-2s)}\int_{\{y:|y|\leq R\}}f_s(y)\,dy\\
			&\leq \Big(\frac{|x|}2\Big)^{-(N-2s)}\int_{\{y:|y|\leq R\}}(|y|^{-\gamma}*|u_s|^2)u_s(y)\,dy\\
			&\leq \Big(\frac{|x|}2\Big)^{-(N-2s)}\int_{\{y:|y|\leq R\}}\Big[1+(|y|^{-\gamma}*|u_s|^2)^{\frac{2N}\gamma}+|y|^{-\gamma}*|u_s|^2|u_s(y)|^2\Big]\,dy\\
			&\leq C(N,s,\gamma,V_0, u_s)|x|^{-(N-2s)}.
		\end{split}
	\end{equation*}

	This inequality and  \eqref{usw} yield the conclusion.
	\quad $\Box$

\bigskip

Now we establish the nonexistence result.

{\bf Proof of  Theorem \ref{p1}.}
By Lemma \ref{lema}, we know that the function  $\psi=(1+|x|^2)^{-\frac{N-2s_0}{2}}$ satisfies
\[
S_s^H=\frac{\int_{\mathbb{R}^N}|(-\Delta)^{\frac {s_0}2}\psi(x)|^2\,dx}{D_\gamma(\psi,\psi)^{\frac12}}.
\]
For $\varepsilon>0$, set $\psi_\varepsilon(x)=\psi( x/\varepsilon)$. Then,
\begin{equation}\label{sh}
\begin{split}
S_{s,\gamma}^V\leq I_{s,\gamma}^V(\psi_\varepsilon)&=
\frac{\varepsilon^{N-2s}\int_{\mathbb{R}^N}|(-\Delta)^{\frac {s}2}\psi(x)|^2\,dx
+\varepsilon^N\int_{\mathbb{R}^N}V(\varepsilon x)|\psi(x)|^2\,dx}
{\varepsilon^{N-2s_0}D_\gamma(\psi,\psi)^{\frac12}}\\
&=
\frac{\varepsilon^{2s_0-2s}\int_{\mathbb{R}^N}|(-\Delta)^{\frac {s}2}\psi(x)|^2\,dx
+\varepsilon^{2s_0}\int_{\mathbb{R}^N}V(\varepsilon x)|\psi(x)|^2\,dx}
{D_\gamma(\psi,\psi)^{\frac12}}\\
&\leq
\frac{\varepsilon^{2s_0-2s}\int_{\mathbb{R}^N}|(-\Delta)^{\frac {s}2}\psi(x)|^2\,dx
+\varepsilon^{2s_0}V_\infty\int_{\mathbb{R}^N}|\psi(x)|^2\,dx}
{D_\gamma(\psi,\psi)^{\frac12}}.
\end{split}
\end{equation}

If $s\in (0,s_0)$, letting $\varepsilon\to 0$ in \eqref{sh}, we obtain  that
$S_{s,\gamma}^V=0$. It means that  there exists no  minimizer  of  $S_{s,\gamma}^V$ in $H^s(\mathbb{R}^N)\setminus \{0\}$.

If $s=s_0$, we deduce from the inequality
\[
S_{s_0,\gamma}^V\leq\frac{\int_{\mathbb{R}^N}|(-\Delta)^{\frac {s_0}2}\psi(x)|^2\,dx
+\varepsilon^{2s_0}\int_{\mathbb{R}^N}V(\varepsilon x)|\psi(x)|^2\,dx}
{D_\gamma(\psi,\psi)^{\frac12}}\leq S_s^H+C\varepsilon^{2s_0}
\]
that $S_{s_0,\gamma}^V\leq S_{s_0}^H$ by letting $\varepsilon\to 0$.

We claim that there exists no  minimizer of  $S_{s,\gamma}^V$ in $H^s(\mathbb{R}^N)\setminus \{0\}$.
Suppose on the contrary that there exists $u\in H^s(\mathbb{R}^N)\setminus \{0\}$ such that $S_{s_0,\gamma}^V=I_{s_0,\gamma}^V(u)$. Then, the inequality
\[
I_{s_0,\gamma}^V(u)>I_{s_0,\gamma}^0(u)+V_0\|u\|_2^2>I_{s_0,\gamma}^0(u)\geq S_{s_0}^H
\]
yields a contradiction.\quad $\Box$

\bigskip

	\section{Asymptotic behavior of $u_s$ as $s\to s_0$}
	
	\smallskip
	
	For $0<s_0=\gamma/4<s<1$, we know from Theorem \ref{thm1} that there exists a nonnegative ground state solution $u_s$  of \eqref{eq:1.1} such that
\begin{equation}\label{s4-1}
		\begin{split}
			D_\gamma(u_s,u_s)=\|u\|_{s,V}^2=\Big(S_{s,\gamma}^V\Big)^2.
		\end{split}
	\end{equation}
In this section, we study the asymptotic behavior of $u_s$  as $s\downarrow s_0$,  which is described in Theorem \ref{thm2}.
We divide the proof into several steps, each one is in a subsection

	In the sequel, we denote by $C>0$ a positive constant independent of $s$, $k$, and $x$.
	
	\subsection{Blow up analysis}\ \
	
\bigskip
	
	In this subsection, we prove $(i)$ of Theorem \ref{thm2}.  We begin with estimates of $S_{s,\gamma}^V$ and $u_s$.
	
	\begin{Lemma}\label{lem1}
		For any $s\in (s_0,1)$, we have
		
		$(i)$ $S_{s,\gamma}^V\leq C$;
		
		\smallskip

		$(ii)$ $\|u_s\|_{H^{s_0}(\mathbb{R}^N)}^2\leq 2\|u_s\|_{H^{s}(\mathbb{R}^N)}^2\leq C$;

		\smallskip
		$(iii)$ $\|u_s\|_{2_{s_0}^*}\leq C$,\quad  $\|u_s\|_{2_{s}^*}\leq C$;
		
		$(iv)$ $\big\||\cdot|*|u_s|^2\big\|_{\frac {2N}\gamma}\leq C$.
	\end{Lemma}
	\begin{proof}
		Choose $\varphi\in C_c^\infty(\mathbb{R}^N)$ such that $D_\gamma(\varphi,\varphi)=1$. By Lemma \ref{lemf},
		\[S_{s,\gamma}^V\leq I_{s,\gamma}^V(\varphi)= \|\varphi\|_{s,V}^2\leq 2\|\varphi\|_{H^1(\mathbb{R}^N)}^2+\|\varphi\|_{1,V}^2,\]
		$(i)$ holds true. $(ii)$  follows from  \eqref{svh1}, \eqref{s4-1}, $(i)$ and Lemma \ref{lemf}.
		$(iii)$  is a consequence of  Lemma \ref{lemc} and $(ii)$. By Lemma \ref{lemb2} and $(iii)$, we see that $(iv)$ is true.
	\end{proof}

	\begin{Lemma}\label{lem3}One has
		
		$(i)$ $\liminf_{s\downarrow s_0}I_{s,\gamma}^0(u_s)\geq S_{s_0}^H$;
		
		\smallskip
		$(ii)$ $\limsup_{s\downarrow s_0}I_{s,\gamma}^V(u_s)\leq S_{s_0}^H$;
		
		\smallskip
		$(iii)$ $\lim_{s\downarrow s_0}I_{s,\gamma}^V(u_s)=\lim_{s\downarrow s_0}D_\gamma(u_s,u_s)^{\frac12}=\lim_{s\downarrow s_0}\|u_s\|_{s,V}= S_{s_0}^H$.
		
		\smallskip
		$(iv)$	$\limsup_{s\downarrow s_0}\|u_s\|_{\dot{H}^{s_0}(\mathbb{R}^N)}\leq S_{s_0}^H$.
	\end{Lemma}
	\begin{proof}
		$(i)$ By Lemmas \ref{lema} and \ref{lemf}, and \eqref{s4-1},  we have
		\begin{equation*}
			\begin{split}
				I_{s,\gamma}^0(u_s)
				&=\frac{\|u_s\|_{\dot{H}^s(\mathbb{R}^N)}^2}{D_\gamma(u_s,u_s)^{\frac12}}\\
				&\geq \frac{\|u_s\|_{\dot{H}^{s_0}(\mathbb{R}^N)}^{\frac{2s}{s_0}}\|u_s\|_2^{-2(\frac{s}{s_0}-1)}}{D_\gamma(u_s,u_s)^{\frac12}}\\
				&=\Big(\frac{\|u_s\|_{\dot{H}^{s_0}(\mathbb{R}^N)}^2}{D_\gamma(u_s,u_s)^{\frac12}}\Big)^{\frac s{s_0}}D_\gamma(u_s,u_s)^{\frac12(\frac s{s_0}-1)}\|u_s\|_2^{-2(\frac{s}{s_0}-1)}\\
				&\geq \Big(S_{s_0}^H\Big)^{\frac s{s_0}}\|u_s\|_{s,V}^{\frac s{s_0}-1} \Big(C\|u_s\|_{s,V}^2\Big)^{-(\frac s{s_0}-1)}\\
				&= C^{-(\frac s{s_0}-1)}\|u_s\|_{s,V}^{-(\frac s{s_0}-1)}\Big(S_{s_0}^H\Big)^{\frac s{s_0}}\\
			\end{split}
		\end{equation*}
Since $s>s_0$, by $(ii)$ of Lemma \ref{lem1} and \eqref{svh1}, we obtain
\[
I_{s,\gamma}^0(u_s)\geq C^{-(\frac s{s_0}-1)}\Big(S_{s_0}^H\Big)^{\frac s{s_0}},
\]
 the conclusion follows by letting $s\downarrow s_0$.

\bigskip
		
		$(ii)$ By Lemma \ref{lema}, we know that
		\[
		\varphi_{s_0}(x)=\big(1+4\pi^2|x|^2\big)^{-\frac{N-2s_0}2}
		\]
		satisfies
		\[
		S_{s_0}^H=\frac{\int_{\mathbb{R}^N}|(-\Delta)^{\frac {s_0}2}\varphi_{s_0}|^2\,dx}{D(\varphi_{s_0},\varphi_{s_0})^{\frac12}}.
		\]
		
		By  \cite[Proposition 6.1.5]{G2009}, we have
		\begin{equation}\label{uste01}
0<\hat{\varphi}_{s_0}(\xi)\leq \tilde{C}(N,s_0)e^{-\frac{|\xi|}{2}} \quad \text{if} \quad
		|\xi|\geq 2
		\end{equation}
		and
		\begin{equation}\label{uste02}
0<\hat{\varphi}_{s_0}(\xi)\leq \tilde{C}(N,s_0)(1+|\xi|^{-2s_0}) \quad \text{if} \quad
		|\xi|\leq 2.
		\end{equation}
		
		Set ${\varphi}_{s_0,\varepsilon}(x)={\varphi}_{s_0}(x/\varepsilon)$ with $\varepsilon>0$. By \eqref{uste01}, \eqref{uste02} and the Lebesgue dominated convergence theorem, we have for $s\downarrow s_0$ that
		\begin{equation*}
			\begin{split}
				\int_{\mathbb{R}^N}|(-\Delta)^{\frac {s}2}\varphi_{s_0,\varepsilon}|^2\,dx&=
				\varepsilon^{N-2s}\int_{\mathbb{R}^N}|\xi|^{2s}|\hat{\varphi}_{s_0}(\xi)|^2\,d\xi\\
				&=\varepsilon^{N-2s_0}
				\Big(\int_{\mathbb{R}^N}|\xi|^{2s_0}|\hat{\varphi}_{s_0}(\xi)|^2\,d\xi+o(1)\Big)\\
				&=\varepsilon^{N-2s_0}
				\Big(\int_{\mathbb{R}^N}|(-\Delta)^{\frac {s_0}2}\varphi_{s_0}|^2\,dx+o(1)\Big).	
			\end{split}
		\end{equation*}
		
		It follows that
		\[
		D_\gamma(\varphi_{s_0,\varepsilon},\varphi_{s_0,\varepsilon})
		=\varepsilon^{2N-4s_0}D_\gamma(\varphi_{s_0},\varphi_{s_0}),
		\]
		and by assumption $(V)$,
		\[
		\int_{\mathbb{R}^N}V|\varphi_{s_0,\varepsilon}|^2\,dx\leq V_\infty\int_{\mathbb{R}^N}|\varphi_{s_0,\varepsilon}|^2\,dx=V_\infty\varepsilon^N\int_{\mathbb{R}^N}|\varphi_{s_0}|^2\,dx.
		\]
		Hence, we have as $s\downarrow s_0$ that
		\begin{equation*}
			\begin{split}
				S_{s,\gamma}^V&\leq I_{s,\gamma}^V(\varphi_{s_0,\varepsilon})\\
				&\leq \frac{\varepsilon^{N-2s_0}
					\Big(\int_{\mathbb{R}^N}|(-\Delta)^{\frac {s_0}2}\varphi_{s_0}|^2\,dx+o(1)\Big)+
					V_\infty\varepsilon^N\int_{\mathbb{R}^N}|\varphi_{s_0}|^2\,dx
				}{\varepsilon^{N-2s_0}D_\gamma(\varphi_{s_0},\varphi_{s_0})^{\frac12}}\\
				&=S_{s_0}^H+o(1)
				+\frac{V_\infty\int_{\mathbb{R}^N}|\varphi_{s_0}|^2\,dx}{D_\gamma(\varphi_{s_0},\varphi_{s_0})^{\frac12}}\varepsilon^{2s_0}.
			\end{split}
		\end{equation*}
		That is,
 \[	\lim_{s\downarrow s_0}S_{s,\gamma}^V\leq S_{s_0}^H+\frac{V_\infty\int_{\mathbb{R}^N}|\varphi_{s_0}|^2\,dx}{D_\gamma(\varphi_{s_0},\varphi_{s_0})^{\frac12}}\varepsilon^{2s_0}.\]

		Since $\varepsilon$ is arbitrary, the assertion follows.

\bigskip
		
		$(iii)$ follows from $(i)$, $(ii)$ and \eqref{s4-1}.

\bigskip
		
		$(iv)$ By Lemmas \ref{lemf} and \ref{lem1}, we have
		\[
		\|u_s\|_{\dot{H}^{s_0}(\mathbb{R}^N)}^2\leq \|u_s\|_{\dot{H}^{s}(\mathbb{R}^N)}^{\frac{2s_0}s}
		\|u_s\|_2^{2(1-\frac {s_0}s)}\leq \|u_s\|_{s,V}^{\frac{2s_0}s}C^{2(1-\frac {s_0}s)}.
		\]
		Letting $s\downarrow s_0$ and applying $(iii)$, we obtain the result.
	\end{proof}

\bigskip
	
	Next, we show $\lim_{s\downarrow s_0}\|u_s\|_\infty=+\infty$, i.e. $\{u_s\}$ blows up as $s\downarrow s_0$.
	\begin{Lemma}\label{lem2}
		For any $s\in (s_0,1)$, one has $\|u_s\|_\infty\geq C>0$.	
	\end{Lemma}
	\begin{proof}
		Using Lemma \ref{lemb}, we have
		\[D_\gamma(u_\gamma,u_\gamma)\leq C(N,\gamma)\|u_s\|_{2_{s_0}^*}^4,\]
		where  $2_{s_0}^*=\frac{4N}{N-\gamma}$ since $s_0=\gamma/4$.
		
		The fact $N\geq 4> 4s_0$ implies   $2<2_{s_0}^*<4$.  By Lemma \ref{lem1}$(iii)$, we get
		$$D_\gamma(u_\gamma,u_\gamma)\leq C\|u_s\|_{2_{s_0}^*}^{2_{s_0}^*}.$$
By \eqref{s4-1},
		\[\int_{\mathbb{R}^N}\Big[C|u_s|^{2_{s_0}^*-2}-V_0\Big]|u_s|^2\,dx\geq \int_{\mathbb{R}^N}\Big[|x|^{-\gamma}*|u_s|^2-V\Big]|u_s|^2\,dx=\int_{\mathbb{R}^N}|(-\Delta)^{\frac s2}u(x)|^2\,dx\geq0,\]
		which yields the result.	
	\end{proof}

\bigskip

	\begin{Lemma}\label{lem4}
		We have  $\lim_{s\downarrow s_0}\|u_s\|_\infty=+\infty$.
	\end{Lemma}
	
	\begin{proof}
		Suppose by a contradiction that there exists a sequence $\{s_k\}$ with $s_k\downarrow s_0$ as $k\to\infty$ such that $\|u_{s_k}\|_\infty\leq C$. By $(ii)$ of Lemma \ref{lem1},
		\[
		|x|^{-\gamma}*|u_{s_k}|^2\leq \|u_{s_k}\|_\infty^2\int_{\{y:|x-y|\leq 1\}}|x-y|^{-\gamma}\,dy+\int_{\{y:|x-y|\geq 1\}}|u_{s_k}(y)|^2\,dy\leq C.
		\]
		Taking $s_1=s_0$ and $f=(|x|^{-\gamma}*|u_{s_k}|^2)u_{s_k}-Vu_{s_k}$ in Lemma \ref{lemd} , we deduce by \eqref{eq:1.1} that
		\[
		\|u_{s_k}\|_{C^{0,\alpha}(\mathbb{R}^N)}\leq C
		\]
		for some $\alpha\in(0,s_0)$.
		Let $\tilde{u}_{s_k}(x)=u_{s_k}(x+x_{s_k})$,  where $x_{s_k}$ is a maximum point of $u_{s_k}$. Then
		\[
		\|\tilde{u}_{s_k}\|_{C^{0,\alpha}(\mathbb{R}^N)}\leq C.
		\]
		We know from  Lemma \ref{lem1} that
		\[\|\tilde{u}_{s_k}\|_{H^{s_0}(\mathbb{R}^N)}\leq C,\quad \Big\||\cdot|^{-\gamma}*|\tilde{u}_{s_k}|^2\Big\|_{\frac{2N}{\gamma}}\leq C.\]
		By Lemma \ref{lem2}, we have $\tilde{u}_{s_k}(0)=\|{u}_{s_k}\|_\infty\geq C>0$.
		
		Hence we may assume that
		$\tilde{u}_{s_k}\to \tilde{u}\neq 0$ in $C_{loc}(\mathbb{R}^N)$,
		$\tilde{u}_{s_k}\rightharpoonup \tilde{u}$ weakly in $H^{s_0}(\mathbb{R}^N)$,
		$|x|^{-\gamma}*|\tilde{u}_{s_k}|^2\rightharpoonup |x|^{-\gamma}*|\tilde{u}|^2$ weakly in $L^{\frac{2N}\gamma}(\mathbb{R}^N)$ as $k\to\infty$.
		Moreove, $\tilde{u}_{s_k}$ satisfies
		\begin{equation}\label{uste}
			(-\Delta)^{s_k}\tilde{u}_{s_k}+V(x+x_{s_k})\tilde{u}_{s_k}=(|x|^{-\gamma}*|\tilde{u}_{s_k}|^2)\tilde{u}_{s_k}.
		\end{equation}
		
		By \eqref{s4-1} and assumption $(V)$, we obtain
		\[
		D_\gamma(\tilde{u}_{s_k},\tilde{u}_{s_k})=\|\tilde{u}_{s_k}\|_{s, V(\cdot,x_{s_k})}^2\geq
		\|(-\Delta)^{\frac{s_k}{2}}\tilde{u}_{s_k}\|_2^2+V_0\|\tilde{u}_{s_k}\|_2^2.
		\]
		$(i)$ and $(iii)$ of Lemma \ref{lem3} then yield
		\begin{equation*}
			\begin{split}
				S_{s_0}^H&=\lim_{k\to\infty}D_{\gamma}(\tilde{u}_{s_k},\tilde{u}_{s_k})^{\frac12}\\
				&=\lim_{k\to\infty}\frac{\|\tilde{u}_{s_k}\|_{s, V(\cdot,x_{s_k})}^2}{D_{\gamma}(\tilde{u}_{s_k},\tilde{u}_{s_k})^{\frac12}}\\
				&\geq \liminf_{k\to\infty}I_{s,\gamma}^0(\tilde{u}_{s_k})
				+\liminf_{k\to\infty}\frac{V_0|\tilde{u}_{s_k}\|_2^2}
				{D_\gamma(\tilde{u}_{s_k},\tilde{u}_{s_k})^{\frac12}}\\
&= \liminf_{k\to\infty}I_{s,\gamma}^0({u}_{s_k})
				+\liminf_{k\to\infty}\frac{V_0|\tilde{u}_{s_k}\|_2^2}
				{D_\gamma(\tilde{u}_{s_k},\tilde{u}_{s_k})^{\frac12}}\\
				&\geq S_{s_0}^H+\frac{V_0|\tilde{u}\|_2^2}
				{S_{s_0}^H},
			\end{split}
		\end{equation*}
		which implies $\tilde{u}=0$, a contradiction.
	\end{proof}
	
\bigskip

	Let
	\[v_s=\mu_s^{\alpha_s} u_s(\mu_s x+x_s),\]
	where
	\[\mu_s^{\alpha_s}=\|u_s\|_\infty^{-1}=u_s(x_s)^{-1},\quad \alpha_s=\frac{N+2s-\gamma}{2}.\]
	Then, $\lim_{s\downarrow s_0}\mu_s=0$, $0\leq v_s\leq 1$ and $v_s$ solves
	\begin{equation}\label{vse}
		(-\Delta)^sv_s+\mu_s^{2s}V(\mu_sx+x_s)v_s=(|x|^{-\gamma}*|v_s|^2)v_s
	\end{equation}
	with
	\begin{equation}\label{vs0}
		v_s(0)=\|v_s\|_\infty=1.	
	\end{equation}	
	
	Next, we give some estimate for $v_s$.
	\begin{Lemma}\label{lem4a}For $s>s_0$ close to $s_0$, we have
		
		\smallskip	
		$(i)$ 	$\|v_s\|_{\dot{H}^s(\mathbb{R}^N)}^2\leq C$, 	$\|v_s\|_{2_s^*}^{2_s^*}\leq C$;
		
		\smallskip
		$(ii)$ 	$\|v_s\|_{\dot{H}^{s_0}(\mathbb{R}^N)}^2\leq C$, 	$\|v_s\|_{2_{s_0}^*}^{2_{s_0}^*}\leq C$;
		
		\smallskip
		$(iii)$ $\Big\||\cdot|^{-\gamma}*|v_s|^2v_s\Big\|_{\frac {2N}\gamma}^{\frac {2N}\gamma}\leq C$.
		
		\smallskip
		$(iv)$ $0\leq \limsup_{s\downarrow s_0}|\cdot|^{-\gamma}*|v_s|^2\leq C$.
		
	\end{Lemma}

	\begin{proof}
		$(i)$	Since $\lim_{s\downarrow s_0}\mu_s=0$, $s>s_0$, by $(ii)$ of Lemma \ref{lem1}, for $s>s_0$ close to $s_0$, we have
		\[
		\|v_s\|_{\dot{H}^s(\mathbb{R}^N)}^2=\mu_s^{4(s-s_0)}\|u_s\|_{\dot{H}^s(\mathbb{R}^N)}^2\leq \|u_s\|_{\dot{H}^s(\mathbb{R}^N)}^2\leq C.
		\]
		By Lemma \ref{lemc},
$$\|v_s\|_{2_s^*}^{2_s^*}\leq C\|v_s\|_{\dot{H}^s(\mathbb{R}^N)}^{2_s^*}\leq C.$$

The conclusion in $(i)$ follows.

\bigskip

Since $\|v_s\|_{\dot{H}^{s_0}(\mathbb{R}^N)}^2=\mu_s^{2(s-s_0)}\|u_s\|_{\dot{H}^{s_0}(\mathbb{R}^N)}^2$, the proof of $(ii)$ is similar to $(i)$.
\bigskip		
		
$(iii)$ Since $0\leq v_s\leq 1$ and $\lim_{s\downarrow s_0}\mu_s=0$, we deduce   from $(iv)$ of Lemma \ref{lem1} for $s>s_0$ close to $s_0$ that
		\begin{equation*}
			\begin{split}
				\Big\||\cdot|^{-\gamma}*|v_s|^2v_s\Big\|_{\frac {2N}\gamma}^{\frac {2N}\gamma}
				\leq \Big\||\cdot|^{-\gamma}*|v_s|^2\Big\|_{\frac {2N}\gamma}^{\frac {2N}\gamma}
				=\mu_s^{\frac{4(s-s_0)}{\gamma}N}	\Big\||\cdot|^{-\gamma}*|u_s|^2\Big\|_{\frac {2N}\gamma}^{\frac {2N}\gamma}\leq C.
			\end{split}
		\end{equation*}
\bigskip
		
		$(iv)$ Noting $0\leq v_s\leq 1$, by   $(ii)$ we get
		\begin{equation*}
			\begin{split}
				0&\leq |x|^{-\gamma}*|v_s|^2\\
				&=\Big(\int_{\{y: |x-y|\leq 1\}}+\int_{\{y:|x-y|\geq 1\}}\Big)\frac{|v_s(y)|^2}{|x-y|^\gamma}\,dy\\
				&\leq \int_{\{y:|x-y|\leq 1\}}|x-y|^{-\gamma}\,dy+
				\Big(\int_{\{y:|x-y|\geq 1\}}|v_s|^{2_{s_0}^*}\Big)^{\frac2{2_{s_0}^*}}
				\Big(\int_{\{y:|x-y|\geq 1\}}|x-y|^{-2N}\,dy\Big)^{1-\frac2{2_{s_0}^*}}\\
				&\leq \int_{\{x:|x|\leq 1\}}|x|^{-\gamma}\,dx
				+\Big(\int_{\mathbb{R}^N}|v_s|^{2_{s_0}^*}\Big)^{\frac2{2_{s_0}^*}}
				\Big(\int_{\{x:|x|\geq 1\}}|x|^{-2N}\,dy\Big)^{1-\frac2{2_{s_0}^*}}\leq C.
			\end{split}
		\end{equation*}
	\end{proof}

\bigskip
	
	Now we study  the asymptotic behavior of $v_s$ as $s\downarrow s_0$.
	\begin{Lemma}\label{lem4b}There exists $C(N,s_0)>0$ such that
		\[
		\lim_{s\downarrow s_0} v_s(x)=v(x):=\Big(1+\frac{|x|^2}{C(N,s_0)}\Big)^{-\frac{N-2s_0}{2}}
		\quad {\rm in} \quad \dot{H}^{s_0}(\mathbb{R}^N)\cap L^{2_{s_0}^*(\mathbb{R}^N)}\cap C_{loc}(\mathbb{R}^N)
		\]
		and
		\[
		\lim_{s\downarrow s_0}\mu_s^{s-s_0}=1.
		\]
	\end{Lemma}
	\begin{proof}
		Let $\{s_k\}$ be a sequence such that $s_k\downarrow s_0$ as $k\to \infty$.
Then $\lim_{k\to\infty}\mu_{s_k}=0$.
By Lemma \ref{lem4a}, we may assume that
		\begin{equation*}
			\begin{split}
				&v_{s_k}\rightharpoonup v \quad \text{weakly in }\quad \dot{H}^{s_0}(\mathbb{R}^N);\\
				&v_{s_k}\rightharpoonup v \quad \text{weakly in }\quad L^{2_{s_0}^*}(\mathbb{R}^N);\\
				&(|x|^{-\gamma}*|v_s|^2)v_s \rightharpoonup (|x|^{-\gamma}*|v|^2)v \quad \text{weakly in }\quad L^{\frac{2N}{\gamma}}(\mathbb{R}^N);\\
				&|\xi|^{s_k}\hat{v}_{s_k} \rightharpoonup |\xi|^{s_0}\hat{v}\quad \text{weakly in } \quad L^{2}(\mathbb{R}^N)
			\end{split}
		\end{equation*}
		as $k\to\infty$.
		
		Let  $\varphi\in C_{c}^\infty(\mathbb{R}^N)$. Since
		$$|\xi|^{2s_k}|\hat{\varphi}|^2\leq (1+|\xi|^2)|\hat{\varphi}|^2\in L^1(\mathbb{R}^N),$$
 we have
		$$|\xi|^{s_k}\hat{\varphi}\to |x|^{s_0}\hat{\varphi}\quad {\rm in}\quad L^{2}(\mathbb{R}^N)$$
as $k\to\infty$.
		Therefore, by \eqref{vse} and $\lim_{k\to\infty}\mu_{s_k}=0$,
		\[
		\int_{\mathbb{R}^N}|\xi|^{2s_0}\hat{v}\hat{\varphi}\,d\xi=\int_{\mathbb{R}^N}(|x|^{-\gamma}*|v|^2)v\varphi\,dx,
		\]
		namely, $v$ solves
		\begin{equation}\label{ve}
			(-\Delta)^{s_0}v=(|x|^{-\gamma}*|v|^2)v.
		\end{equation}
		
		By \eqref{vse}, $0\leq v_s\leq 1$, $\lim_{k\to\infty}\mu_{s_k}=0$ and $(iv)$ of Lemma \ref{lem4a}, we have
$$\|(-\Delta)^{s_k}v_{s_k}\|_\infty^2\leq C.$$
 Lemma \ref{lemd} yields that there exists $\alpha\in (0,s_0)$ such that $\|v_{s_k}\|_{C^{0,\alpha}(\mathbb{R}^N)}\leq C$. Hence, we may assume that
		\[
		v_{s_k}\to v \quad \text{in}\quad C_{loc}(\mathbb{R}^N)\quad \text{as}\quad k\to\infty.
		\]
		The fact  $v_{s_k}(0)=1$ implies $v(0)=1$.
		
		By the definition of $S_{s_0}^H$, $\lim_{k\to\infty}\mu_{s_k}=0$ and $(iii)$ of Lemma \ref{lem3}, we obtain
		\begin{equation*}
			\begin{split}
				S_{s_0}^H&\leq \frac{\|v\|_{\dot{H}^{s_0}(\mathbb{R}^N)}^2}{D_\gamma(v,v)^{\frac12}}
				=D_\gamma(v,v)^{\frac12}=\|v\|_{\dot{H}^{s_0}(\mathbb{R}^N)}\\
				&\leq \liminf_{k\to\infty}D_\gamma(v_k,v_k)^{\frac12}\\
				&=\liminf_{k\to\infty}\mu_{s_k}^{2(s_k-s_0)}D_\gamma(u_k,u_k)^{\frac12}\\
				&=\liminf_{k\to\infty}\mu_{s_k}^{2(s_k-s_0)}S_{s_0}^H\\
				&\leq \limsup_{k\to\infty}\mu_{s_k}^{2(s_k-s_0)}S_{s_0}^H\\
				&\leq S_{s_0}^H,
			\end{split}
		\end{equation*}
		which yields that

$$\lim_{k\to\infty}\mu_{s_k}^{s_k-s_0}=1$$
 and
		\[
		\frac{\|v\|_{\dot{H}^{s_0}(\mathbb{R}^N)}^2}{D_\gamma(v,v)^{\frac12}}
		=D_\gamma(v,v)^{\frac12}=\|v\|_{\dot{H}^{s_0}(\mathbb{R}^N)}=S_{s_0}^H.
		\]
		By \eqref{ve}, $v(0)=1$ and Lemma \ref{lema}, there exists a $C(N,s_0)>0$ such that
		\begin{equation}\label{vgs}
			v(x)=\Big(1+\frac{|x|^2}{C(N,s_0)}\Big)^{-\frac{N-2s_0}{2}}.
		\end{equation}

		 By $(iv)$ of Lemma \ref{lem3} and $\lim_{k\to\infty}\mu_{s_k}^{s_k-s_0}=1$, we obtain
		\[
		\limsup_{k\to\infty}\|v_{s_k}\|_{\dot{H}^{s_0}(\mathbb{R}^N)}
		=\limsup_{k\to\infty}\mu_{s_k}^{2(s_k-s_0)}\|u_{s_k}\|_{\dot{H}^{s_0}(\mathbb{R}^N)}\leq S_{s_0}^H.
		\]

		Using that $$\liminf_{k\to\infty}\|v_{s_k}\|_{\dot{H}^{s_0}(\mathbb{R}^N)}\geq
		\|v\|_{\dot{H}^{s_0}(\mathbb{R}^N)}= S_{s_0}^H,$$
 one has $$\lim_{k\to\infty}\|v_{s_k}\|_{\dot{H}^{s_0}(\mathbb{R}^N)}=\|v\|_{\dot{H}^{s_0}(\mathbb{R}^N)}.$$ This implies
		$v_{s_k}\to v$ strongly in $\dot{H}^{s_0}(\mathbb{R}^N)$ and then in $L^{2_{s_0}^*}(\mathbb{R}^N)$ by Lemma \ref{lemc}. Since $\{s_k\}$ is arbitrary, the conclusion follows.
	\end{proof}

\bigskip

	\bigskip
	
	\subsection{Proof of $(i)$ in Theorem \ref{thm2}}
	
	\bigskip
	By Lemma \ref{lem4b}, to complete the proof of Theorem \ref{thm2}$(i)$, we need to show
	\[\lim_{s\downarrow s_0}	v_s=v=\Big(1+\frac{|x|^2}{C(N,s_0)}\Big)^{-\frac{N-2s_0}{2}}
	\quad \text{in }\quad   C(\mathbb{R}^N).\]

	Since $\lim_{s\downarrow s_0}	v_s=v$ in $C_{loc}(\mathbb{R}^N)$, the conclusion will be proved once we establish a uniformly decaying law of $v_s$ with respect to $s$ at infinity. To this aim, we  give some estimates on $\tilde{u}_s$ and $v_s$.

	Denote $K_{f}(x)=|x|^{-\gamma}\ast |f|^2$.
	
	\begin{Lemma}\label{lem6a}
		We have $K_{v_s}\to K_v$ strongly in $L^{\frac{N}{2s_0}}(\mathbb{R}^N)$ as $s\downarrow s_0$.	
	\end{Lemma}
	\begin{proof}
		 By the H\"{o}lder inequality,  Lemmas \ref{lemb2} and \ref{lem4a}, we deduce
		\begin{equation*}
			\begin{split}
				&\|K_{v_s}-K_v\|_{\frac{N}{2s_0}}^{\frac{N}{2s_0}}\\
				&=\int_{\mathbb{R}^N}\Big\{|x|^{-\gamma}*[(|v_s|+|v|)(|v_s|-|v|)]\Big\}^{\frac{N}{2s_0}}
				\,dx\\
				&\leq \int_{\mathbb{R}^N}\Big\{|x|^{-\gamma}*[2(|v_s|^2+|v|^2)]\Big\}^{\frac{N}{4s_0}}
				\Big\{|x|^{-\gamma}*(|v_s-v|^2)\Big\}^{\frac{N}{4s_0}}\,dx\\
				&\leq\Bigg\{\int_{\mathbb{R}^N}\Big\{|x|^{-\gamma}*[2(|v_s|^2+|v|^2)]\Big\}^{\frac{N}{2s_0}}
				\,dx\Bigg\}^{\frac12}\Bigg\{\int_{\mathbb{R}^N}\Big\{|x|^{-\gamma}*|v_s-v|^2\Big\}^{\frac{N}{2s_0}}
				\,dx\Bigg\}^{\frac12}\\
				&\leq C\Big[\|v_s\|_{2_{s_0}^*}^{\frac{N}{s_0}}+\|v\|_{2_{s_0}^*}^{\frac{N}{s_0}}\Big]^{\frac12}
				\|v_s-v\|_{2_{s_0}^*}^{\frac{N}{2s_0}}.
			\end{split}
		\end{equation*}
		The assertion follows since Lemma \ref{lem4b} implies $v_s\to v$ in $L^{2_{s_0}^*}(\mathbb{R}^N)$ as $s\downarrow s_0$.
	\end{proof}
	
	Recall that  $\tilde{u}_s(x)=u_s(x+x_s)$, where $x_{s}$ is a maximum point of $u_{s}$.
	\begin{Lemma}\label{lem6}Let $R>0$. Then
		
		\smallskip
		$(i)$ $\lim_{s\downarrow s_0} \int_{\{x:|x|\geq R\}}|\tilde{u}_s|^{2_{s_0}^*}\,dx=0$;
		
		\smallskip
		$(ii)$ $\lim_{s\downarrow s_0} \int_{\{x:|x|\geq R\}}\Big||x|^{-\gamma}*|\tilde{u}_s|^2\Big|^{\frac{N}{2s_0}}\,dx=0$.
	\end{Lemma}
	\begin{proof}
		$(i)$ Since $\mu_s\to 0$, $\lim_{s\downarrow s_0}\mu_s^{s-s_0}=1$, and by Lemma \ref{lem4b}, $v_s\to v$ in $L^{2_{s_0}^*}(\mathbb{R}^N)$ as $s\downarrow s_0$, the conclusion follows from
		\begin{equation*}
			\begin{split}
				\int_{\{x:|x|\geq R\}}|\tilde{u}_s|^{2_{s_0}^*}\,dx
				&=	\mu_s^{-\frac{2N}{N-2s_0}(s-s_0)}\int_{\{x:|x|\geq R\mu_s^{-1}\}}|\tilde{v}_s|^{2_{s_0}^*}\,dx\\
				&\leq C\Big[
				\int_{\{x:|x|\geq R\mu_s^{-1}\}}|\tilde{v}_s-v|^{2_{s_0}^*}\,dx
				+\int_{\{x:|x|\geq R\mu_s^{-1}\}}|v|^{2_{s_0}^*}\,dx\Big].
			\end{split}
		\end{equation*}
	\bigskip	
		
		$(ii)$ 			
		By Lemma \ref{lem6a} and  $\mu_s\to 0$ as $s\downarrow s_0$, we obtain
		\begin{equation*}
			\begin{split}
				\int_{\{x:|x|\geq R\}}\Big||x|^{-\gamma}*|\tilde{u}_s|^2\Big|^{\frac{N}{2s_0}}\,dx
				&=\mu_s^{-\frac N{s_0}(s-s_0)}\int_{\{x:|x|\geq R\mu_s^{-1}\}}|K_{v_s}|^{\frac{N}{2s_0}}\,dx\\
				&\leq C\Big[\|K_{v_s}-K_v\|_{\frac{N}{2s_0}}^{\frac{N}{2s_0}}+
				\int_{\{x:|x|\geq R\mu_s^{-1}\}}|K_v|^{\frac{N}{2s_0}}\,dx\Big]
				\to 0
			\end{split}
		\end{equation*}
as $s\downarrow s_0.$
	\end{proof}

\bigskip

	\begin{Lemma}\label{lem7}
		For  $R>0$, there holds
		\[
		\lim_{s\downarrow s_0}\max_{x\in B_R(0)^c}\tilde{u}_s(x)=0.
		\]
	\end{Lemma}
	\begin{proof}
		Choose $r>0$ sufficiently small such that $|B_r(0)|<1$ and $r<R$. By \eqref{uste} and $V\geq 0$, we have
		\begin{equation}\label{ust2}
		(-\Delta)^{s}\tilde{u}_s=[|x|^{-\gamma}*|\tilde{u}_s|^2-V(x+x_s)]\tilde{u}_s\leq (|x|^{-\gamma}*|\tilde{u}_s|^2)\tilde{u}_s.
		\end{equation}

		For $|x|\geq R$. we deduce from Lemma \ref{lem6} that
		\[\Big\||x|^{-\gamma}*|\tilde{u}_s|^2\Big\|_{L^{\frac{N}{2s}}(B_r(x))}
		\leq \Big\||x|^{-\gamma}*|\tilde{u}_s|^2\Big\|_{L^{\frac{N}{2s_0}}(B_r(x))}
		\leq \Big\||x|^{-\gamma}*|\tilde{u}_s|^2\Big\|_{L^{\frac{N}{2s_0}}(B_{R-r}(0)^c)}\to 0
		\]
as $s\downarrow s_0.$
		Using Lemmas \ref{leme2} and Lemma \ref{lem1} $(iii)$, for $s>s_0$ close to $s_0$, we have
		\[
		\|\tilde{u}_s\|_{L^{\frac{(2_s^*)^2}{2}}(B_{r/2}(x))}\leq C\|\tilde{u}_s\|_{2_s^*}=C\|{u}_s\|_{2_s^*}\leq C.
		\]
		Since
		\[
		\lim_{s\downarrow s_0}\frac{(2_s^*)^2}{2}=\frac{2N^2}{(N-2s_0)^2}>2_{s_0}^*,
		\]
for $s>s_0$ close to $s_0$, we can choose
		$p\in (2_{s_0}^*, \frac{(2_{s_0}^*)^2}{2})$ such that $p<\frac{(2_s^*)^2}{2}$ and
		\[
		\|\tilde{u}_s\|_{L^{p}(B_{r/2}(x))}\leq \|\tilde{u}_s\|_{L^{\frac{(2_s^*)^2}{2}}(B_{r/2}(x))}
		\leq C.
		\]
		Therefore, by Lemma \ref{lemb2}, we obtain
		\begin{equation}\label{lemb2a}
			\int_{B_r(x)}\Big(\int_{B_r(x)}\frac{|\tilde{u}_s(z)|^2}{|y-z|^\gamma}\,dz\Big)^qdy
			\leq C\|\tilde{u}_s\|_{L^p(B_r(x))}^{2q}\leq C,
		\end{equation}
		where $1+\frac 1q=\frac2p+\frac\gamma N$.
		Hence, we get via Lemma \ref{lem1} that
		\begin{equation*}
			\begin{split}
				&\int_{B_{r/2}(x)}(|y|^{-\gamma}*|\tilde{u}_s|^2)^q\,dy\\
				&=	\int_{B_{r/2}(x)}\Big(\int_{\{z:|y-z|\leq \frac r2\}}\frac{|\tilde{u}_s(z)|^2}{|y-z|^\gamma}\,dz
				+\int_{\{z:|y-z|\geq \frac r2\}}\frac{|\tilde{u}_s(z)|^2}{|y-z|^\gamma}\,dz
				\Big)^q\,dy\\
				&\leq 2^q \Big\{	\int_{B_{r/2}(x)}\Big[\int_{B_{r}(x)}\frac{|\tilde{u}_s(z)|^2}{|y-z|^\gamma}\,dz	\Big]^q\,dy
				+\int_{B_{r/2}(x)}\Big[(\gamma/2)^{-\gamma}\int_{\mathbb{R}^N}|\tilde{u}_s(z)|^2\,dz\Big]^q\,dz\Big\}\\
				&\leq 2^q \Bigg\{	\int_{B_{r}(x)}\Big[\int_{B_{r}(x)}\frac{|\tilde{u}_s(z)|^2}{|y-z|^\gamma}\,dz	\Big]^q\,dy
				+C\Bigg\}	\leq C.
			\end{split}
		\end{equation*}

Note that $q>\frac{N}{2s_0}$.    Applying  Lemma \ref{leme} to \eqref{ust2} with $s_1=s_0$, $s_2>s_0$ close to $s_0$, $a(y)=|y|^{-\gamma}*|\tilde{u}_s|^2-V(y+x_s)$ and $b=0$,   we have for $x\in B_R(0)^c$ that
		\[
		\tilde{u}_s(x)\leq \sup_{B_{r/4}(x)}\tilde{u}_s(y)\leq C\inf_{B_{r/4}(x)}\tilde{u}_s(y)
		\leq C\|\tilde{u}_s\|_{L^{2_{s_0}^*}(B_{r/4}(x))}\leq C\|\tilde{u}_s\|_{L^{2_{s_0}^*}(B_{R-r/4}(0)^c)},
		\]
		which yields
		\[
		\max_{B_R(0)^c}\tilde{u}_s(x)\leq C\|\tilde{u}_s\|_{L^{2_{s_0}^*}(B_{R-r/4}(0)^c)}.
		\]
		This and  Lemma \ref{lem6} implies the desired result.
	\end{proof}

\bigskip

	\begin{Lemma}\label{lem8} Let $R>0$. For $s>s_0$ close to $s_0$, we have
		
		\smallskip	
		$(i)$	$\tilde{u}_s(x)\leq C|x|^{-N}, \quad {\rm if}\quad|x|\geq R$;
		
		\smallskip	
		$(ii)$  $|x|^{-\gamma}*|\tilde{u}_s|^2\leq C|x|^{-\gamma}, \quad {\rm if}\quad|x|\geq R$;
		
		\smallskip	
		$(iii)$  $|x|^{-\gamma}*|{v}_s|^2\leq C\mu_s^{2s-\gamma}|x|^{-\gamma}, \quad {\rm if}\quad |x|\geq \mu_s^{-1}R$.
	\end{Lemma}
	\begin{proof}$(i)$
		By Lemma \ref{lem7}, it suffices to prove that there exists $R$ large enough such that the result holds. Let $|x|\geq R$. Then by Lemma \ref{lem1},
		\begin{equation*}
			\begin{split}
				|x|^{-\gamma}*|\tilde{u}_s|^2&=
				\int_{\{y:|x-y|\leq R/2\}}\frac{|\tilde{u}_s(y)|^2}{|x-y|^\gamma}\,dy
				+\int_{\{y:|x-y|\geq R/2\}}\frac{|\tilde{u}_s(y)|^2}{|x-y|^\gamma}\,dy\\
				&\leq \max_{\{y:|x-y|\leq R/2\}}|\tilde{u}_s(y)|^2 \int_{\{y:|x-y|\leq R/2\}}|x-y|^{-\gamma}\,dy
				+\big(\frac R2\big)^{-\gamma}\int_{\mathbb{R}^N}|\tilde{u}_s(y)|^2\,dy\\
				&\leq C(N)\max_{\{y:|y|\geq R/2\}}|\tilde{u}_s(y)|^2 R^{N-\gamma}+C_1R^{-\gamma},
			\end{split}
		\end{equation*}
		where  $C_1>0$ is independent of $s$, $x$ and $R$.
		
		Choose $R>0$ large enough such that
		\[
		C_1R^{-\gamma}<\frac14V_0.
		\]
		By Lemma \ref{lem7}, for $s>s_0$ close to $s_0$, we have
		\[
		C(N)\max_{\{y:|y|\geq R/2\}}|\tilde{u}_s(y)|^2 R^{N-\gamma}<\frac14V_0.
		\]
		Rewrite \eqref{uste} as
		\[
		(-\Delta)^s\tilde{u}_s+\frac12V_0\tilde{u}_s=(|x|^{-\gamma}*|\tilde{u}_s|^2-V(x+x_s)+\frac12V_0)\tilde{u}_s:=h_s(x).
		\]
		Then, for $|x|\geq R$, $h_s(x)\leq (|x|^{-\gamma}*|\tilde{u}_s|^2-\frac12V_0)\tilde{u}_s\leq 0$, we have
		\[
		\tilde{u}_s=\int_{\mathbb{R}^N}K_{s,\frac12V_0}(x-y)h_s(y)\,dy\leq \int_{\{y:|y|\leq R\}}K_{s,\frac12V_0}(x-y)h_s(y)\,dy,
		\]
		where $K_{s,\frac12V_0}>0$ is the fundamental solution of
		\[
		(-\Delta)^s\tilde{u}_s+\frac12V_0\tilde{u}_s=0.
		\]
		It is known from Lemma C.1 $(ii)$ in \cite{FLS} that
		\[K_{s,\frac12V_0}\leq C(N,s_0)|x|^{-N}, \quad \text{if}\quad s_0<s<1.
		\]
		Hence, if $|x|\geq 2R$, we derive from Lemma \ref{lem1} that
		\begin{equation*}
			\begin{split}
				\tilde{u}_s(x)&\leq C|x|^{-N}\int_{|y|\leq R}h_s(y)\,dy\\
				&\leq  C|x|^{-N}\int_{\{y:|y|\leq R\}}(|x|^{-\gamma}*|\tilde{u}_s|^2)\tilde{u}_s\,dy\\
				&\leq C|x|^{-N}\int_{\{y:|y|\leq R\}}(|x|^{-\gamma}*|\tilde{u}_s|^2)(1+\tilde{u}_s^2)\,dy\\
				&\leq C|x|^{-N}\int_{\{y:|y|\leq R\}}\Big[1+(|x|^{-\gamma}*|\tilde{u}_s|^2)^{\frac N{2s_0}}+(|x|^{-\gamma}*|\tilde{u}_s|^2)\tilde{u}_s^2\Big]\,dy\\
				&\leq C|x|^{-N}.
			\end{split}
		\end{equation*}
		This yields the result of $(i)$.

\bigskip
		
		$(ii)$ Let $|x|\geq R$. By Lemma \ref{lem1} and $(i)$, we obtain that
		\begin{equation*}
			\begin{split}
				|x|^{-\gamma}*|\tilde{u}_s|^2
				&=\Big(\int_{\{y:|y|\leq |x|/2\}}+\int_{\{y:|x|/2\leq |y|\leq 2|x|\}}+\int_{\{y:|y|\geq 2|x|\}}\Big)\frac{|\tilde{u}_s(y)|^2}{|x-y|^\gamma}\,dy\\
				&\leq \Big(\frac{|x|}{2}\Big)^{-\gamma}\int_{\{y:|y|\leq |x|/2}\}\}|\tilde{u}_s(y)|^2\,dy
				+C|x|^{-2N}\int_{\{y:|x-y|\leq 3|x|\}}|x-y|^{-\gamma}\,dy\\
				&+|x|^{-\gamma}\int_{\{y:|y|\geq 2|x|\}}|\tilde{u}_s(y)|^2\,dy\\
				&\leq \Big(\frac{|x|}{2}\Big)^{-\gamma}\int_{\mathbb{R}^N}|\tilde{u}_s(y)|^2\,dy
				+C(N)|x|^{-N-\gamma}+
				C|x|^{-\gamma}\int_{\{y:|y|\geq 2R\}}|y|^{-2N}\,dy\\
				&\leq C\Big(1+R^{-N}+\int_{\{y:|y|\geq 2R\}}|y|^{-2N}\,dy\Big)|x|^{-\gamma}\leq
				C|x|^{-\gamma}.
			\end{split}
		\end{equation*}

\bigskip
		$(iii)$ By $(ii)$, if $|x|\geq \mu_s^{-1}R$, we have
		\[
		|x|^{-\gamma}*|v_s|^2=\mu_s^{2s}(|\cdot|^{-\gamma}*|\tilde{u}_s|^2)(\mu_sx)
		\leq C\mu_s^{2s-\gamma}|x|^{-\gamma}.
		\]

The proof is complete.
	\end{proof}

\bigskip
	
	Now, we establish the uniformly decaying law of $v_s$ with respect to $s$ at  infinity.
	\begin{Lemma}\label{lem9} For $s>s_0$ close to $s_0$, we have
		\[
		v_s(x)\leq C(1+|x|^2)^{-\frac{N-2s}{2}},\quad x\in \mathbb{R}^N.
		\]	
	\end{Lemma}
	\begin{proof}
		Let $\bar{v}_s$ be  the Kelvin transform of $v_s$, namely,
		\[
		\bar{v}_s=\frac1{|x|^{N-2s}}v_s\Big(\frac{x}{|x|^2}\Big).
		\]
We infer from the  well known result
		\[
		(-\Delta)^s\bar{v}_s=\frac1{|x|^{N+2s}}(-\Delta)^s{v}_s\Big(\frac{x}{|x|^2}\Big)
		\]
and \eqref{vse} that $\bar{v}_s$ satisfies
		\[
		(-\Delta)^s\bar{v}_s+\mu_s^{2s}|x|^{-4s}V(\mu_s\frac x{|x|^2}+x_s)\bar{v}_s=|x|^{-4s}(|x|^{-\gamma}*|v_s|^2)\bar{v}_s.
		\]

		Thus,
		\begin{equation}\label{vbse}
			(-\Delta)^s\bar{v}_s\leq a(x)\bar{v}_s,
		\end{equation}
		where
		\[
		a(x)=|x|^{-4s}(|\cdot|^{-\gamma}\ast|v_s|^2)(x/|x|^2)=|x|^{\gamma-4s}\Big[|x|^{-\gamma}\ast(|y|^{\gamma-4s}|\bar{v}_s|^2)\Big].
		\]
		
Since $0\leq v_s\leq 1$, the proof will be finished once we show  $\bar{v}_s\leq C$.  We divide the  proof  into two  steps.
		
		\bigskip
		Step 1. We prove that for $s>s_0$ close to $s_0$, there exists $r>0$ such that
		\[
		\|\bar{v}_s\|_{L^{\frac{(2_s^*)^2}{2}}\big(B_{\frac r2}(0)\big)}\leq C.
		\]
		
\bigskip

		By Lemma \ref{lem8}, if $|x|\leq \mu_s$, we have
		\[
		K_{v_s}(x/|x|^2)=|\cdot|^{-\gamma}\ast|v_s|^2)(x/|x|^2)\leq C\mu_s^{2s-\gamma}|x|^\gamma
		\]
		and
		\[a(x)\leq C\mu_s^{2s-\gamma}|x|^{\gamma-4s}.\]
		For a fixed  $q\in(0,\frac N{s_0})$, if $s>s_0$ close to $s_0$, we derive from $\lim_{s\downarrow s_0}\mu_s=0$  and
$$\lim_{s\downarrow s_0}(\gamma-6s)q+2N=2N-2s_0q>0$$
that
		\begin{equation}\label{llem9}
			\begin{split}
				\int_{\{x:|x|\leq \mu_s^2\}}a(x)^q\,dx&\leq C\mu_s^{(2s-\gamma)q}\int_{\{x:|x|\leq \mu_s^2\}}|x|^{(\gamma-4s)q}\,dx\\
				&\leq C\mu_s^{(\gamma-6s)q+2N}<\delta(N)/2,
			\end{split}
		\end{equation}
		where $\delta(N)$ is given in Lemma \ref{leme2}.
		
		Furthermore, by Lemma \ref{lem6a} and $\lim_{s\downarrow s_0}\mu_s^{s-s_0}=1$, if $s>s_0$ is sufficiently close $s_0$, there exists $r>0$ small enough and independent of $s$, such that
		\begin{equation*}
			\begin{split}
				&\int_{\{x:\mu_s^2\leq|x|\leq r\}}a(x)^{\frac{N}{2s}}\,dx\\
				&=\int_{\{x:\mu_s^2\leq|x|\leq r\}}|x|^{-2N}K_{v_s}(x/|x|^2)^{\frac{N}{2s}}\,dx\\
				&=\int_{\{x:r^{-1}\leq |x|\leq \mu_s^{-2}\}}K_{v_s}(x)^{\frac{N}{2s}}\,dx\\
				&\leq \Big(\int_{\{x:r^{-1}\leq |x|\leq \mu_s^{-2}\}}K_{v_s}(x)^{\frac{N}{2s_0}}\,dx\Big)^{\frac{s_0}{s}}
				\Big(\int_{\{x:|x|\leq \mu_s^{-2}\}}\,dx\Big)^{1-\frac{s_0}{s}}\\
				&\leq  C\mu_s^{-\frac{2N}{s}(s-s_0)}\Big[\int_{ \{x:|x|\geq r^{-1}\}}|K_{v_s}(x)-K_v(x)|^{\frac{N}{2s_0}}\,dx+\int_{\{x: |x|\geq r^{-1}\}}|K_v(x)|^{\frac{N}{2s_0}}\,dx\Big]^{\frac{s_0}{s}}\\
&<\delta(N)/2.		\\			
			\end{split}
		\end{equation*}
		Hence
		\[
		\int_{|x|\leq r}a(x)^{\frac{N}{2s}}\,dx<\delta(N).
		\]

		Now Lemma \ref{leme2} and Lemma \ref{lem4a}$(i)$ yield
\[
\|\bar{v}_s\|_{L^{\frac{(2_s^*)^2}{2}}\big(B_{\frac r2}(0)\big)}
\leq C\|\bar{v}_s\|_{2_s^*}=C\|{v}_s\|_{2_s^*}
\leq C.
\]
		
		\bigskip

		Step 2. We prove that for $s>s_0$ close to $s_0$,
		\[
		\|\bar{v}_s\|_{\infty}\leq C.
		\]

\smallskip

		Rewrite $a(x)$ as
		\[
		a(x)=|x|^{\gamma-4s}\Big[|x|^{-\gamma}\ast(|y|^{\gamma-4s}|\bar{v}_s|^2)\Big].
		\]
		
		We start with estimating $|x|^{-\gamma}\ast(|y|^{\gamma-4s}|\bar{v}_s|^2)$ for $|x|\geq \mu_s^2$.

		If $|x|\geq \mu_s^2$ and  $s>s_0$ close to
		$s_0$, we derive from Lemma \ref{lem8} and $\lim_{s\downarrow s_0}\mu_s=0$ that
		\begin{equation*}
			\begin{split}
				&\int_{\{y:|y|\leq \mu_s^2/2\}}|x-y|^{-\gamma}|y|^{\gamma-4s}|\bar{v}_s(y)|^2\,dy\\
				&=\int_{\{y:|y|\leq \mu_s^2/2\}}|x-y|^{-\gamma}|y|^{\gamma-2N}\Big|{v}_s(\frac{y}{|y|^2})\Big|^2\,dy\\
				&\leq C\mu_s^{-2\gamma+2\alpha_s}\int_{\{y:|y|\leq \mu_s^2/2\}}|y|^{\gamma-2N}\Big|{\tilde{u}}_s(\frac{\mu_s y}{|y|^2})\Big|^2\,dy\\
				&\leq C\mu_s^{-2\gamma+2\alpha_s-2N} \int_{\{y:|y|\leq \mu_s^2/2\}}|y|^{\gamma}\,dy\\
				&\leq C\mu_s^{2\alpha_s}\leq C.
			\end{split}
		\end{equation*}
		
		Hence,  for $s>s_0=\gamma/4$  close to $s_0$, we deduce by the fact  $\lim_{s\downarrow s_0}\mu_s^{s-s_0}=1$ that

		\begin{equation}\label{zh2}
			\begin{split}
				&\int_{\{y:\mu_s^2/2\leq |y|\leq r/2\}}|x-y|^{-\gamma}|y|^{\gamma-4s}|\bar{v}_s(y)|^2\,dy\\
				&\leq \Big(\frac{\mu_s^2}{2}\Big)^{\gamma-4s}\int_{\{y: |y|\leq r/2\}}|x-y|^{-\gamma}|\bar{v}_s(y)|^2\,dy\\
				&\leq C\int_{\{y: |y|\leq r/2\}}|x-y|^{-\gamma}|\bar{v}_s(y)|^2\,dy
			\end{split}
		\end{equation}
		
		Since $0\leq v_s\leq 1$, if $|x|\leq r/4$, we derive from the inequality
        \[
		|\bar{v}_s(y)|=\frac1{|y|^{N-2s}}v_s\Big(\frac{y}{|y|^2}\Big)\leq \frac1{|y|^{N-2s}}
		\]
that
		\begin{equation}\label{zh3}
			\begin{split}
				&\int_{\{y:|y|\geq r/2\}}|x-y|^{-\gamma}|y|^{\gamma-4s}|\bar{v}_s(y)|^2\,dy\\
				&\leq \int_{\{y:|y|\geq r/2\}}\Big(\frac {|y|}{2}\Big)^{-\gamma}|y|^{\gamma-4s}|\bar{v}_s(y)|^2\,dy\\
				&\leq 2^{\gamma}\int_{\{y:|y|\geq r/2\}}|y|^{-4s}|\bar{v}_s(y)|^2\,dy\\
				&\leq 2^{\gamma}\int_{\{y:|y|\geq r/2\}}|y|^{-2N}\,dy.	
			\end{split}
		\end{equation}
		
		Hence, for $s>s_0$ close to $s_0$, it is valid that
		\begin{equation}\label{zh}
			|x|^{-\gamma}\ast(|y|^{\gamma-4s}|\bar{v}_s|^2)\leq C+\int_{ |y|\leq r/2}|x-y|^{-\gamma}|\bar{v}_s(y)|^2\,dy.
		\end{equation}
		
		Choose $p$ such that $p<\frac{1}{2}(2_{s_0}^*+\frac{(2_{s_0}^*)^2}{2})$ and
		\[
		\frac{s_0}{N}<\frac{1}{q_1}:=\frac 2p+\frac{4s_0}N-1<\frac {2s_0}N.
		\]
Then,
\begin{equation}\label{zh1}
			\begin{split}
				\int_{\{x:\mu_s^2\leq |x|\leq \frac r4\}} a(x)^{q_1}\,dx
				&=\int_{\{x:\mu_s^2\leq |x|\leq \frac r4\}}	|x|^{(\gamma-4s)q_1}\Big[|x|^{-\gamma}\ast(|y|^{\gamma-4s}|\bar{v}_s|^2)\Big]^{q_1}\,dx\\	
				&\leq 	\mu_s^{2(\gamma-4s){q_1}}\int_{\{x:\mu_s^2\leq |x|\leq \frac r4\}}\Big[|x|^{-\gamma}\ast(|y|^{\gamma-4s}|\bar{v}_s|^2)\Big]^{q_1}\,dx.\\
\end{split}
		\end{equation}

For $s>s_0$ close to  $s_0$, using the fact  $\lim_{s\downarrow s_0}\mu_s^{s-s_0}=1$, $\gamma=4s_0$ and $\lim_{s\downarrow s_0}\mu_s=0$ , we derive from Step 1, \eqref{zh}, \eqref{zh1} and Lemma \ref{lemb2} that
		\begin{equation*}
			\begin{split}
				\int_{\{x:\mu_s^2\leq |x|\leq \frac r4\}} a(x)^{q_1}\,dx
				&\leq 	C\int_{\{x:\mu_s^2\leq |x|\leq \frac r4\}}
				\Big[C+\int_{\{y: |y|\leq r/2\}}|x-y|^{-\gamma}|\bar{v}_s(y)|^2\,dy\Big]^{q_1}\,dx\\
				&\leq C\Bigg\{C_1+\int_{\{x: |x|\leq \frac r2\}}	\Big[\int_{\{y: |y|\leq r/2\}}|x-y|^{-\gamma}|\bar{v}_s(y)|^2\,dy\Big]^{q_1}\,dx\Bigg\}	\\
				&\leq C\||\bar{v}_s(y)|^2\|_{L^{p/2}(B_{\frac r2}(0))}^{q_1}+C\\
				&\leq C\|\bar{v}_s(y)\|_{L^{p}(B_{\frac r2}(0))}^{2q_1}+C\\
&\leq C\|\bar{v}_s\|_{L^{\frac{(2_s^*)^2}{2}}\big(B_{\frac r2}(0)\big)}^{2q_1}+C\leq C.
			\end{split}
		\end{equation*}

		This together with \eqref{llem9} with $q=q_1$ yields
		\[
		\int_{\{x:|x|\leq \frac r4\}} a(x)^{q_1}\,dx\leq C.
		\]
				
For $s>s_0$ close to $s_0$, we deduce from Lemma \ref{lem4a}$(iv)$ that	
\[
a(x)=|x|^{-4s}(|\cdot|^{-\gamma}\ast|v_s|^2)(x/|x|^2)\leq C|x|^{-4s}.
\]	
Then
\begin{equation*}
\begin{split}
\int_{\{x:|x|\geq \frac r4\}}a(x)^{q_1}\,dx
\leq C\int_{\{x:|x|\geq \frac r4\}}|x|^{-4sq_1}\,dx
\leq C\int_{\{x:|x|\geq \frac r4\}}(|x|^{-4s_0q_1}+|x|^{-4q_1})\,dx
\leq C.
\end{split}
\end{equation*}
Hence
\[
\int_{\mathbb{R}^N}a(x)^{q_1}\,dx\leq C.
\]
Using Lemma \ref{lem:A.3} with $t=q_1$,  we derive from  \eqref{vbse} and Lemma \ref{lem4a} $(i)$ that
\[
\|\bar{v}_s\|_\infty\leq C\|\bar{v}_s\|_{2_s^*}=C\|{v}_s\|_{2_s^*}\leq C.
\]

The proof is complete.

	\end{proof}
	{\bf Completion of Proof of Theorem \ref{thm2}$(i)$.}
\smallskip

By Lemma \ref{lem4b}, $v_s\to v$ in $C_{loc}(\mathbb{R}^N)$ as $s\downarrow s_0$, where $$v=\Big(1+\frac{|x|^2}{C(N,s_0)}\Big)^{-\frac{N-2s_0}{2}}.$$
 From Lemma \ref{lem9}, for $s>s_0$ close to $s_0$, there exists $\delta\in (0,1-s_0)$ such that
$$v_s\leq C(1+|x|^2)^{-\frac{N-2s_0-2\delta}{2}}.$$ Thus,  $v_s\to v$ in $C(\mathbb{R}^N)$ as $s\downarrow s_0$.\quad $\Box$
	
\bigskip
	
	\subsection{Proofs of  Theorem \ref{thm2}$(ii)$ and $(iii)$}\ \

\bigskip

	{\bf Proof of  Theorem \ref{thm2} $(ii)$:}
	Let $y_0$ be minimum point of $V$. By assumption $(V)$, one has
	\[
	V(y_0)=\inf_{x\in\mathbb{R}^N}V(x)>V_0>0.
	\]
	Inspired by \cite{WZ}, choose $u_s(\cdot+x_s-y_0)$ as a test function for $S_{s,\gamma}^V$ such that
	\[
	I_{s,\gamma}^V(u_s)\leq I_{s,\gamma}^V(u_s(\cdot+x_s-y_0))
	\]
	and
	\[
	\int_{\mathbb{R}^N}V(x)|u_s(x)|^2\,dx\leq \int_{\mathbb{R}^N}V(x)|u_s(x+x_s-y_0)|^2\,dx
	=\mu_s^{-2s+\gamma}\int_{\mathbb{R}^N}V(\mu_sx+y_0)|v_s(x)|^2\,dx
	\]

	Since $N\geq4$, by assumption $(V)$ and Lemma \ref{lem9},  for $s>s_0$ close to  $s_0$, there exists $\delta\in (0,1-s_0)$ such that
	\[
	0\leq V(\mu_sx+y_0)|v_s(x)|^2\leq C(1+|x|^2)^{-(N-2s_0-2\delta)}\in L^1(\mathbb{R}^N).
	\]
	
	Then by the Lebesgue dominated convergence theorem, we obtain
	\begin{equation}\label{bpp1}
\begin{split}
		\int_{\mathbb{R}^N}V(x)|u_s(x)|^2\,dx&\leq \mu_s^{-2s+\gamma}\Big[V(y_0)\int_{\mathbb{R}^N}|v(x)|^2\,dx+o(1)\Big]\\
&=\mu_s^{-2s+\gamma}\Big[\inf_{x\in\mathbb{R}^N}V(x)\int_{\mathbb{R}^N}|v(x)|^2\,dx+o(1)\Big]
		\quad \text{as}\quad s\downarrow s_0.
\end{split}
	\end{equation}

	For any sequence $\{s_k\}$ with $s_k\downarrow s_0$ as $k\to \infty$, there exists a subsequence, still denoted  by $\{s_k\}$, such that
 $$
 \lim_{k\to\infty}V(x_{s_k})=a\in \Big[\inf_{x\in\mathbb{R}^N}V(x), V_\infty\Big]
 $$
  as $k\to \infty$.  Thus, we derive from
	\[
	\int_{\mathbb{R}^N}V(x)|u_s(x)|^2\,dx=\mu_s^{-2s+\gamma}\int_{\mathbb{R}^N}V(\mu_sx+x_s)|v_s(x)|^2\,dx
	\]
that
	\[
	\int_{\mathbb{R}^N}V(x)|u_{s_k}(x)|^2\,dx=\mu_{s_k}^{-2s_k+\gamma}\Big[a\int_{\mathbb{R}^N}|v(x)|^2\,dx+o(1)\Big]
	\quad \text{as}\quad k\to\infty.
	\]
	This, together with \eqref{bpp1}, yields
$$a\leq \inf_{x\in\mathbb{R}^N}V(x).$$
Thus $a= \inf_{x\in\mathbb{R}^N}V(x)$ and
	\begin{equation*}
		\int_{\mathbb{R}^N}V(x)|u_{s_k}(x)|^2\,dx=\mu_{s_k}^{-2{s_k}+\gamma}\Big[\inf_{x\in\mathbb{R}^N}V(x)\int_{\mathbb{R}^N}|v(x)|^2\,dx+o(1)\Big]
		\quad \text{as}\quad k\to \infty.
	\end{equation*}

	The arbitrariness of $\{s_k\}$ implies that
	\[
	\lim_{s\downarrow s_0}V(x_s)=\inf_{x\in\mathbb{R}^N}V(x)
	\]
	and
	\begin{equation}\label{vugz}
		\int_{\mathbb{R}^N}V(x)|u_{s}(x)|^2\,dx=\mu_{s}^{-2s+\gamma}
		\Big[\inf_{x\in\mathbb{R}^N}V(x)\int_{\mathbb{R}^N}|v(x)|^2\,dx+o(1)\Big]
		\quad \text{as}\quad s\downarrow s_0.
	\end{equation}
	The proof is complete.
	\quad $\Box$

\bigskip
	
	{\bf Proof of  Theorem \ref{thm2} $(iii)$.} We first estimate $\int_{\mathbb{R}^N}x\cdot \nabla V|u_s|^2\,dx$.

\bigskip
	By Theorem \ref{thm2}$(ii)$ and assumption $(V)$, for any sequence $\{s_k\}$ with $s_k\downarrow s_0$ as $k\to \infty$, there exists a subsequence, still denoted by $\{s_k\}$, such that
	$x_{s_k}\to x_0$, where  $x_0$ is a minimum point of $V$. Hence,
	\[
	\lim_{k\to\infty}\nabla V(x_{s_k})= \nabla V(x_0)=0.
	\]
	The arbitrariness of $\{s_k\}$ yields that
	\[
	\lim_{s\downarrow s_0}\nabla V(x_{s})=0.
	\]
	Then by assumption $(V)$ and Lemma \ref{lem9}, we obtain
	\begin{equation}\label{xtvug}
		\int_{\mathbb{R}^N}x\cdot \nabla V|u_s|^2\,dx=
		\mu_s^{\gamma-2s}\int_{\mathbb{R}^N}(\mu_sx+x_s)\cdot \nabla V(\mu_sx+x_s)|v_s|^2\,dx=o(1)\mu_s^{\gamma-2s}
	\end{equation}
as  $s\downarrow s_0.$
	Since $u_s$ is a minimizer of $S_{s,\gamma}^V$,  we have
	\[
	\Big[\frac{d}{dt}I_{s,\gamma}^V(u_s(t\cdot))\Big](1)=0,
	\]
which implies that
	\[
	2(s-s_0)\int_{\mathbb{R}^N}|(-\Delta)^su_s|^2\,dx=2s_0\int_{\mathbb{R}^N}V|u_s|^2\,dx
	+\int_{\mathbb{R}^N}x\cdot \nabla V|u_s|^2\,dx,
	\]
	that is,
	\[
	2(s-s_0)\|u_s\|_{s,V}^2=2s\int_{\mathbb{R}^N}V|u_s|^2\,dx
	+\int_{\mathbb{R}^N}x\cdot \nabla V|u_s|^2\,dx.
	\]
	By Lemma \ref{lem3}, \eqref{vugz} and \eqref{xtvug}, we derive
	\[
	2(s-s_0)\Big[(S_{s_0}^H)^2+o(1)\Big]=2s\mu_{s}^{-2s+\gamma}\Big[\inf_{x\in\mathbb{R}^N}V(x)\int_{\mathbb{R}^N}|v(x)|^2\,dx+o(1)\Big],
	\]
	which, along with the definition of $\mu_s$ and $\lim_{s\downarrow s_0}\mu_s^{s-s_0}=1$, yields the desired result. $\Box$

	\bigskip

	\section{Asymptotic behavior of $u_s$ as $s\uparrow1$}

\bigskip

In this section, we study the asymptotic behavior of $u_s$ as $s\uparrow1$.  We will prove Theorem \ref{thm3}. Let us start with estimates for $S_{s,\gamma}^V$ and $S_{s,\gamma}^{V_\infty}$ as $s\uparrow 1$.

\begin{Lemma}\label{st1} It holds that
\smallskip

$(i)$ $
 \limsup_{s\uparrow1} S_{s,\gamma}^V\leq S_{1,\gamma}^V,\quad \limsup_{s\uparrow1} S_{s,\gamma}^{V_\infty}\leq S_{1,\gamma}^{V_\infty};
$
\smallskip

$(ii)$ $
 \liminf_{s\uparrow1} S_{s,\gamma}^{V_\infty}
 \geq \liminf_{s\uparrow1} S_{s,\gamma}^V\geq S_{s_0}^H>0.
$
\smallskip

$(iii)$
$\liminf_{s\uparrow1} S_{s,\gamma}^{V_\infty}=S_{1,\gamma}^{V_\infty}$.
\end{Lemma}
\begin{proof}
$(i)$
Let $\psi$ be a minimizer of $S_{1,\gamma}^V$. By Lemma \ref{lemf}, we get
\begin{equation*}
\begin{split}
S_{s,\gamma}^V&\leq \frac{\int_{\mathbb{R}^N}|(-\Delta)^{\frac{s}{2}}\psi|^2\,dx
+\int_{\mathbb{R}^N}V|\psi|^2\,dx}{D_\gamma(\psi,\psi)^{\frac12}}\\
&\leq \frac{\|\psi\|_{\dot{H}^1(\mathbb{R}^N)}^{2s}\|\psi\|_{L^2(\mathbb{R}^N)}^{2(1-s)}
+\int_{\mathbb{R}^N}V|\psi|^2\,dx}{D_\gamma(\psi,\psi)^{\frac12}}.
\end{split}
\end{equation*}
Letting $s\uparrow1$, we have   $$\limsup_{s\uparrow1} S_{s,\gamma}^V\leq S_{1,\gamma}^V.$$ Similarly, $$\limsup_{s\uparrow1} S_{s,\gamma}^{V_\infty}\leq S_{1,\gamma}^{V_\infty}$$ is valid.

\bigskip

$(ii)$	
By Lemmas \ref{lema} and \ref{lemf},  for any  $s\in (s_0,1)$, we have
\begin{equation*}
\begin{split}
S_{s,\gamma}^{V_\infty}>S_{s,\gamma}^V&=I_{s,\gamma}^V(u_s)=\frac{\|u_s\|_{s,V}^2}{D_\gamma(u_s,u_s)^{\frac12}}\\
&\geq \min\{\frac12,\frac12V_0\}\frac{2\|u_s\|_{H^s(\mathbb{R}^N)}^2}{D_\gamma(u_s,u_s)^{\frac12}}\\
&\geq\min\{\frac12,\frac12V_0\}\frac{\|u_s\|_{H^{s_0}(\mathbb{R}^N)}^2}{D_\gamma(u_s,u_s)^{\frac12}}\\
&\geq \min\{\frac12,\frac12V_0\}\frac{\|u_s\|_{\dot{H}^{s_0}(\mathbb{R}^N)}^2}{D_\gamma(u_s,u_s)^{\frac12}}\\
&\geq S_{s_0}^H>0,
\end{split}
\end{equation*}
which yields the desired result.

\bigskip

$(iii)$  We know from \cite{AS} that there exists a radial ground state solution $w_s$  of
\begin{equation}\label{vwe}
(-\Delta)^su+V_\infty u=(|x|^{-\gamma}*|u|^2)u,\quad {\rm in} \quad \mathbb{R}^N
\end{equation}
such that  $S_{s,\gamma}^{V_\infty}=I_{s,\gamma}^{V_\infty}(w_s)$
and
\begin{equation}\label{000}
\|w_{s}\|_{s,V_\infty}^2=D_\gamma(w_{s}, w_{s})=\Big(S_{s,\gamma}^{V_\infty}\Big)^2.
\end{equation}
Similar to Lemma \ref{lem1}, for any $s\in (s_1,1)$ with $s_1>s_0$, we can deduce that
\[
\|w_s\|_{H^{s_1}(\mathbb{R}^N)}\leq 2\|w_s\|_{H^{s}(\mathbb{R}^N)}\leq C,\quad \Big\||x|^{-\gamma}*|w_s|^2\Big\|_{\frac{2N}\gamma}\leq C.
\]
For any sequence $\{s_k\}$,  $s_k\uparrow 1$ as $k\to\infty$, there exist a subsequence, still denoted  by $\{s_k\}$,   such that
\begin{equation}\label{111}
\begin{split}
&w_{s_k}\rightharpoonup w\quad \text{weakly in } \quad H^{s_1}(\mathbb{R}^N)\\
&|x|^{-\gamma}*|w_s|^2\rightharpoonup |x|^{-\gamma}*|w|^2
\quad \text{weakly in } \quad L^{\frac{2N}\gamma}(\mathbb{R}^N)\\
&|\xi|^{s_k}\hat{w}_{s_k}\rightharpoonup |\xi|\hat{w}
\quad \text{weakly in } \quad L^2(\mathbb{R}^N).
\end{split}
\end{equation}
Since $2<2_{s_0}^*<2_{s_1}^*$, the embedding $L^{2_{s_0}^*}(\mathbb{R}^N)\hookrightarrow H_r^{s_1}(\mathbb{R}^N)$ is compact,
we may assume
\[
w_{s_k}\to w\quad \text{strongly in } \quad L^{2_{s_0}^*}(\mathbb{R}^N).
\]
By Lemma \ref{lemb},
\[
D_\gamma(w_{s_k}-w,w_{s_k}-w)\leq
C\|w_{s_k}-w\|_{2_{s_0}^*}^{4}\to 0\quad \text{as}\quad k\to\infty.
\]
From the nonlocal Brezis-Lieb Lemma \cite[Lemma 2.2]{BFV}, we have
\begin{equation}\label{dc}
\lim_{k\to\infty}D_\gamma(w_{s_k},w_{s_k})=
\lim_{k\to\infty}D_\gamma(w,w).
\end{equation}
Therefore, \eqref{000} and $(ii)$ yield $D_\gamma(w,w)>0$, so  $w\neq0$.

Since, $w_{s_k}$ solves \eqref{vwe}, we have
\[
\int_{\mathbb{R}^N}|\xi|^{2s_k}\hat{w}_{s_k}\hat{\varphi}\,d\xi
+\int_{\mathbb{R}^N}V_\infty{w}_{s_k}{\varphi}\,dx=
\int_{\mathbb{R}^N}(|x|^{-\gamma}*|w_{s_k}|^2){w}_{s_k}\varphi\,dx,
\]
for any $\varphi\in C_c^\infty(\mathbb{R}^N)$.
We derive by \eqref{111} that
\[
\int_{\mathbb{R}^N}|\xi|^{2}\hat{w}\hat{\varphi}\,d\xi
+\int_{\mathbb{R}^N}V_\infty{w}{\varphi}\,dx=
\int_{\mathbb{R}^N}(|x|^{-\gamma}*|w|^2){w}\varphi\,dx.
\]
Then
\begin{equation}\label{222}
\|w\|_{1,V_\infty}^2=D_\gamma(w, w)=\Big(\frac{\|w\|_{1,V_\infty}^2}{D_\gamma(w, w)^{\frac12}}\Big)^2\geq(S_{1,\gamma}^{V_\infty})^2
\end{equation}
via the definition of $S_{1,\gamma}$.

It follows from \eqref{000}, \eqref{dc} and \eqref{222} that
\[
\liminf_{s\uparrow1}S_{s,\gamma}^{V_\infty}\geq S_{1,\gamma}^{V_\infty},
\]
which and $(i)$ give the result.
\end{proof}

\bigskip

Now we give the proof of Theorem \ref{thm3}.
\smallskip

	{\bf Proof of Theorem \ref{thm3}.} The proof is divided into three  steps.
\bigskip

Step 1. We prove that for any sequence $\{s_k\}$, $s_k\uparrow 1$, there exist a subsequence, still denoted by $\{s_k\}$, and $u\in H^{1}(\mathbb{R}^N)$ such that
\begin{equation}\label{dl0}
\lim_{k\to\infty}D_\gamma(u_{s_k},u_{s_k})=D_\gamma(u,u).
\end{equation}

For any sequence $\{s_k\}$,  $s_k\uparrow 1$ as $k\to\infty$, equation \eqref{s4-1} and Lemma \ref{st1} imply that
there exist a subsequence, still denoted by $\{s_k\}$,  $l>0$ and $l_1\geq 0$ such that
\begin{equation}\label{jx1}
\lim_{k\to\infty}\|u_{s_k}\|_{s_k,V}^2=\lim_{k\to\infty}D_\gamma(u_{s_k},u_{s_k})=
\lim_{k\to\infty}\Big(S_{s_k,\gamma}^V\Big)^2=l>0.
\end{equation}
and
\[
\lim_{k\to\infty}\|u_{s_k}\|_{\dot{H}^{s_k}(\mathbb{R}^N)}^2=
\lim_{k\to\infty}\Big\||\cdot|^{s_k}\hat{u}_{s_k}\Big\|_{L^2(\mathbb{R}^N)}^2
=l_1,
\quad\lim_{k\to\infty}\int_{\mathbb{R}^N}V|u_{s_k}|^2\,dx=l-l_1.
\]
By Lemmas \ref{lemc} and \ref{lem1}, we may assume that as $k\to\infty$,
\begin{equation*}
\begin{split}
&u_{s_k}\rightharpoonup u\quad \text{weakly in } \quad H^{s_0}(\mathbb{R}^N);\\
&u_{s_k}\to u\quad \text{strongly in } \quad L_{loc}^2(\mathbb{R}^N);\\
&|\xi|^{s_k}\hat{u}_{s_k}\rightharpoonup |\xi|\hat{u} \quad \text{weakly in } \quad
L^2(\mathbb{R}^N).
\end{split}
\end{equation*}
Then by the Fatou Lemma, we have $u\in H^1(\mathbb{R}^N)$ and
\begin{equation}\label{lg}
D_\gamma(u,u)\leq \lim_{k\to\infty}D_\gamma(u_{s_k},u_{s_k})= l.
\end{equation}

Let $w_{s_k}=u_{s_k}-u$. We have
\begin{equation*}
\begin{split}
&w_{s_k}\rightharpoonup 0\quad \text{weakly in } \quad H^{s_0}(\mathbb{R}^N);\\
&w_{s_k}\to 0\quad \text{strongly in } \quad L_{loc}^2(\mathbb{R}^N);\\
&|\xi|^{s_k}\hat{w}_{s_k}\rightharpoonup 0 \quad \text{weakly in } \quad
L^2(\mathbb{R}^N).
\end{split}
\end{equation*}
Apparently,
\begin{equation*}
\begin{split}
l&=\lim_{k\to\infty}\|u_{s_k}\|_{s_k,V}^2\\
&=\lim_{k\to\infty}\Big\||\cdot|^{s_k}\hat{u}_{s_k}\Big\|_{L^2(\mathbb{R}^N)}^2
+\lim_{k\to\infty}\int_{\mathbb{R}^N}V|u_{s_k}|^2\,dx\\
&=\lim_{k\to\infty}\Big\||\cdot|^{s_k}\hat{w}_{s_k}\Big\|_{L^2(\mathbb{R}^N)}^2
+\lim_{k\to\infty}\int_{\mathbb{R}^N}V|w_{s_k}|^2\,dx
+\lim_{k\to\infty}\Big\||\cdot|^{s_k}\hat{u}\Big\|_{L^2(\mathbb{R}^N)}^2
+\int_{\mathbb{R}^N}V|u|^2\,dx.\\
\end{split}
\end{equation*}
Since $\lim_{x\to \infty}V(x)=V_\infty$, we obtain
\begin{equation*}
l=\lim_{k\to\infty}\|w_{s_k}\|_{s_k,V_\infty}^2+\lim_{k\to\infty}\|u\|_{s_k,V}^2.
\end{equation*}

On the other hand, 	 the nonlocal Brezis-Lieb Lemma \cite[Lemma 2.2]{BFV} yields
\begin{equation*}
\begin{split}
l=\lim_{k\to\infty}D_\gamma(u_{s_k},u_{s_k})=\lim_{k\to\infty}D_\gamma(w_{s_k},w_{s_k})
+D_\gamma(u,u).
\end{split}
\end{equation*}
Hence
\begin{equation}\label{hx}
\begin{split}
l&\geq \liminf_{k\to\infty}S_{s_k,\gamma}^{V_\infty}D_\gamma(w_{s_k},w_{s_k})^{\frac12}
+\lim_{k\to\infty}S_{s_k,\gamma}^VD_\gamma(u,u)^{\frac12}\\
&= \liminf_{k\to\infty}S_{s_k,\gamma}^{V_\infty}\Big(l-D_\gamma(u,u)\Big)^{\frac12}
+\lim_{k\to\infty}S_{s_k,\gamma}^VD_\gamma(u,u)^{\frac12}.
\end{split}
\end{equation}

\bigskip

Now we prove \eqref{dl0}, namely,
\begin{equation}\label{dl}
D_\gamma(u,u)=l.
\end{equation}

 By \eqref{lg}, we know that  $0\leq D_\gamma(u,u)\leq l$. Now, suppose on the contrary that  $0< D_\gamma(u,u)<l$, we would have by $S_{s_k,\gamma}^{V_\infty}> S_{s_k,\gamma}^{V}$, \eqref{jx1} and Lemma \ref{st1}$(ii)$ that
\begin{equation*}
\begin{split}
l&\geq \liminf_{k\to\infty}S_{s_k,\gamma}^{V}\Big(l-D_\gamma(u,u)\Big)^{\frac12}
+\lim_{k\to\infty}S_{s_k,\gamma}^VD_\gamma(u,u)^{\frac12}\\
&> \liminf_{k\to\infty}S_{s_k,\gamma}^{V}l^{\frac12}=l,
\end{split}
\end{equation*}
which is a contradiction.

If $D_\gamma(u,u)=0$, then by Lemma \ref{st1} $(iii)$ and \eqref{hx},
\[
l\geq \liminf_{k\to\infty}S_{s_k,\gamma}^{V_\infty}l^\frac12
=S_{1,\gamma}^{V_\infty}l^\frac12>S_{1,\gamma}^{V}l^\frac12,
\]
that is,
\[
l=\lim_{k\to\infty}\Big(S_{s_k,\gamma}^V\Big)^2>\Big(S_{1,\gamma}^{V}\Big)^2,
\]
which contradicts  Lemma \ref{st1}$(i)$. Hence, \eqref{dl} holds.

\bigskip

Step 2. We prove that
\[
\lim_{k\to \infty}u_{s_k}=u\quad \text{in}\quad L_{2_{s_0}^*}(\mathbb{R}^N).
\]

From Step 1 and the nonlocal Brezis-Lieb Lemma\cite[Lemma 2.2]{BFV}, we have
\[
\lim_{k\to\infty}D_\gamma(u_{s_k}-u, u_{s_k}-u)=\lim_{k\to\infty}D_\gamma(u_{s_k}, u_{s_k})
-D_\gamma(u, u)=0.
\]
Hence, \cite[Proposition 2.1(4)]{BFV} implies that
\[
\|u_{s_k}-u\|_{2_{s_0}^*}\leq CD_\gamma(u_{s_k}-u, u_{s_k}-u)^{\frac14}\to 0
\quad \text{as}\quad k\to \infty.
\]
In other word,  the conclusion is valid.

\bigskip

Step 3. We prove that $u$ is a ground state solution of \eqref{eq:1.1} with $s=1$.

\smallskip
Similar to that in the proof of  \ref{st1}$(iii)$, we can see that $u$  solves
\[
-\Delta u+Vu=(|x|^{-\gamma}*|u|^2)u.
\]
Then
\[
l=D_\gamma(u,u)=\|u\|_{1,V}^2=\Big(\frac{\|u\|_{1,V}^2}{D_\gamma(u,u)^{\frac12}}\Big)^2\geq
\Big(S_{1,\gamma}^V\Big)^2.
\]

On the other hand, by \eqref{jx1} and Lemma \ref{st1}$(i)$,
\[
l=\lim_{k\to\infty}\|u_{s_k}\|_{s_k,V}^2=\lim_{k\to\infty}D_\gamma(u_{s_k},u_{s_k})=
\lim_{k\to\infty}\Big(S_{s_k,\gamma}^V\Big)^2\leq \Big(S_{1,\gamma}^V\Big)^2.
\]
Hence

\[
l=\lim_{k\to\infty}\Big(S_{s_k,\gamma}^V\Big)^2=\Big(S_{1,\gamma}^V\Big)^2
=\Big(\frac{\|u\|_{1,V}^2}{D_\gamma(u,u)^{\frac12}}\Big)^2.
\]

This yields that  $u$ is a ground state of \eqref{eq:1.1} with $s=1$. $\Box$

\bigskip

	{\bf Data Availability} {There is no data in the paper.}
	
	{\bf Statements and Declarations}  { There is no conflict of interest for all authors.}
	
	{\bf Acknowledgements} {  Jianfu Yang is supported by NNSF of China, No:12171212. Jinge Yang was supported by NNSF of China, No:12361025 and  by Jiangxi
		Provincial Natural Science Foundation, No:20232BAB201002 , 20252BAC230002.
	}

\end{document}